\documentclass{article}
\usepackage[utf8]{inputenc}
\usepackage{amsmath, amssymb, amsthm,mathrsfs,amsfonts}
\usepackage{color}
\usepackage{subcaption}
\usepackage{graphicx}
\usepackage{dsfont}
\usepackage{hyperref}
\usepackage{float}
\usepackage{amssymb,amsmath,amsthm}
\usepackage{amsfonts}
\usepackage{hyperref,url}
\usepackage{subcaption}
\usepackage{graphicx}
\usepackage{array,booktabs}
\usepackage{multirow}
\usepackage{authblk}
\usepackage{xcolor}
\usepackage[export]{adjustbox}

\newtheorem{theorem}{Theorem}[section]
\newtheorem{proposition}[theorem]{Proposition}

\newtheorem{lemma}[theorem]{Lemma}

\newtheorem{example}[theorem]{Example}

\newcommand\extrafootertext[1]{%
    \bgroup
    \renewcommand\thefootnote{\fnsymbol{footnote}}%
    \renewcommand\thempfootnote{\fnsymbol{mpfootnote}}%
    \footnotetext[0]{#1}%
    \egroup
}

\begin{document}

\title{Projected Spread Models}
\date{\vspace{-3em}}

\author[1]{Jung-Chao Ban}
\author[1]{Jyy-I Hong\footnote{Corresponding author. \textit{Email address}: hongjyyi@nccu.edu.tw}}
\author[1]{Cheng-Yu Tsai}
\author[2]{Yu-Liang Wu}
\affil[1]{Department of Mathematical Sciences, National Chengchi University, Taipei 11605, Taiwan, ROC.}
\affil[2]{Department of Mathematical Sciences, P.O. Box 3000, 90014 University of Oulu, Finland}
\maketitle

\begin{abstract}

We present a disease transmission model that considers both explicit and non-explicit factors. This approach is crucial for accurate prediction and control of infectious disease spread. In this paper, we extend the spread model from our previous works \cite{ban2021mathematical,ban2023randomspread, ban2023mathematical, ban2023spread} to a projected spread model that considers both hidden and explicit types. Additionally, we provide the spread rate for the projected spread model corresponding to the topological and random models. Furthermore, examples and numerical results are provided to illustrate the theory.

\extrafootertext{\textit{Keywords}: projected spread model, topological spread model, random spread model, spread rate.}

\end{abstract}

\baselineskip=1.4 \baselineskip

\section{Introduction}

The main theme of the paper is a proposed model of disease transmission that takes into account of both explicit and non-explicit disease-related factors. This model emerges naturally as an approach to explain certain dynamics observed in a pandemic, especially in the wake of the COVID-19 outbreak that shook the world in late 2019.

For a better conception of the behavior of a pandemic, researchers drawn from a diverse array of fields have harnessed tools including (partial) differential equations or machine learning, as demonstrated in works like \cite{alexander2004vaccination, AK-CSF2020, gourley2014mathematical, ou2006spatial, ShuJinWangWu2024JMB, wang2012basic} or random stochastic models, as seen in the works of \cite{Albani2024JMB}, \cite{BG-C2020} and \cite{FA-C2020}, to paint a portrait of stochastic phenomena that are sensitive to parameters and initial conditions. Among all the others, a prevalent class of models classifies each individual as one of the three types: susceptible, infected, and recovered. These models, known as SIR models, are usually studied along with an associated number referred to as a ``basic reproduction number", which is the expected number of infected cases stemming from a single existing case. With this number, one can then quantitatively compare the disease transmissions before and after containment measures are taken \cite{BA-CSF2021, Bichara2023JMB, EMI+-IDM2020, Grundel2022,KS-C2020,LPG+-m2020,MB-S2020,MJK+-C2020,NIE+-MB2020,Stella2022}.

In this article, we extend the authors' previous works \cite{ban2021mathematical, ban2023mathematical, ban2023spread} by complementing each type in the SIR model with a collection of states indicating the current state, usually implicit, of a given individual. We then compare the spread rate of the observed systems with the underlying system.

The motivation of this work came from a very practical problem. When a pandemic occurs, the contagiousness of each infected patient is often related to factors such as the individual's physical condition, the viral load carried by the patient, or the severity of the illness. Therefore, in order to better predict and control the spread rate of an infectious disease, it is crucial to accurately classify patients or virus carriers by certain testing methods and then apply appropriate preventive and control measures accordingly to each class. For convenience, people in the same category are said to be of the same type. Usually, the more precise the classification, the more helpful it is for controlling the epidemic. However, in reality, due to the need for timeliness or because of limitations in scientific technology, the accuracy of testing methods is often constrained. This often leads to gray areas between two classes, causing several different classes to be tested as if they are the same. In such cases, many types with different characteristics are ``hidden" and appear as a same ``explicit" type. To illustrate this, we take Figure \ref{fig: H-E 1-dim} as an example of how the hidden types of patients are transformed into the corresponding explicit types. In this example, each patient belongs to exactly one of six classes labeled by types $a_1, a_2, a_3, a_4, b_1$ and $b_2$ according the viral load each one carries, ranging from high to low. On the top of the figure, from the left to the right, we can see that the first patient is of type $b_2$, the second patient is of type $a_2$, the third patient is of type $a_2$, and so on. A test is applied to each patient; however, this testing method can only detect positive results (labeled as type $a$) once the viral load exceeds a certain level. Assume that only the viral loads carried by the patients of types $a_1, a_2, a_3$ and, $a_4$ exceed the threshold level. Therefore, after testing, patients of types $a_1, a_2, a_3$ and $a_4$ receive test result $a$ and patients of types $b_1$ and $b_2$ receive test result $b$. So, the bottom of the figure represents the corresponding explicit type of each patient after the test. Although not knowing the exact hidden type of each patient is quite common in real-life situations, at least, by knowing their explicit types in this case, patients carrying higher viral loads can be separated from those carrying lower viral loads. Therefore, obtaining the information about the explicit types is also helpful in controlling epidemics. 

To address such an issue in a more formal way, we introduce a mapping (denoted by $\phi$ or $\Phi$) under which each ``hidden" type is projected to its ``explicit" type, and we construct a so-called ``projected spread model" to describe this phenomenon. This mechanism will allow us to study the relationship between the spread rates before and after testing. Moreover, this model can also be applied to the cases in which there is an incubation period between initial infection and first spread. 

This paper is organized as follows. The settings and properties of topological projected spread models and frequent-used notations are introduced in Section \ref{sec: top}. Random projected spread models are dealt with in Section \ref{random models}. Finally, Section \ref{sec: num} consists of examples and experiments that provide numerical evidences for our main theorems.

\begin{figure}[H]
    \centering
    \includegraphics[width=0.8\linewidth]{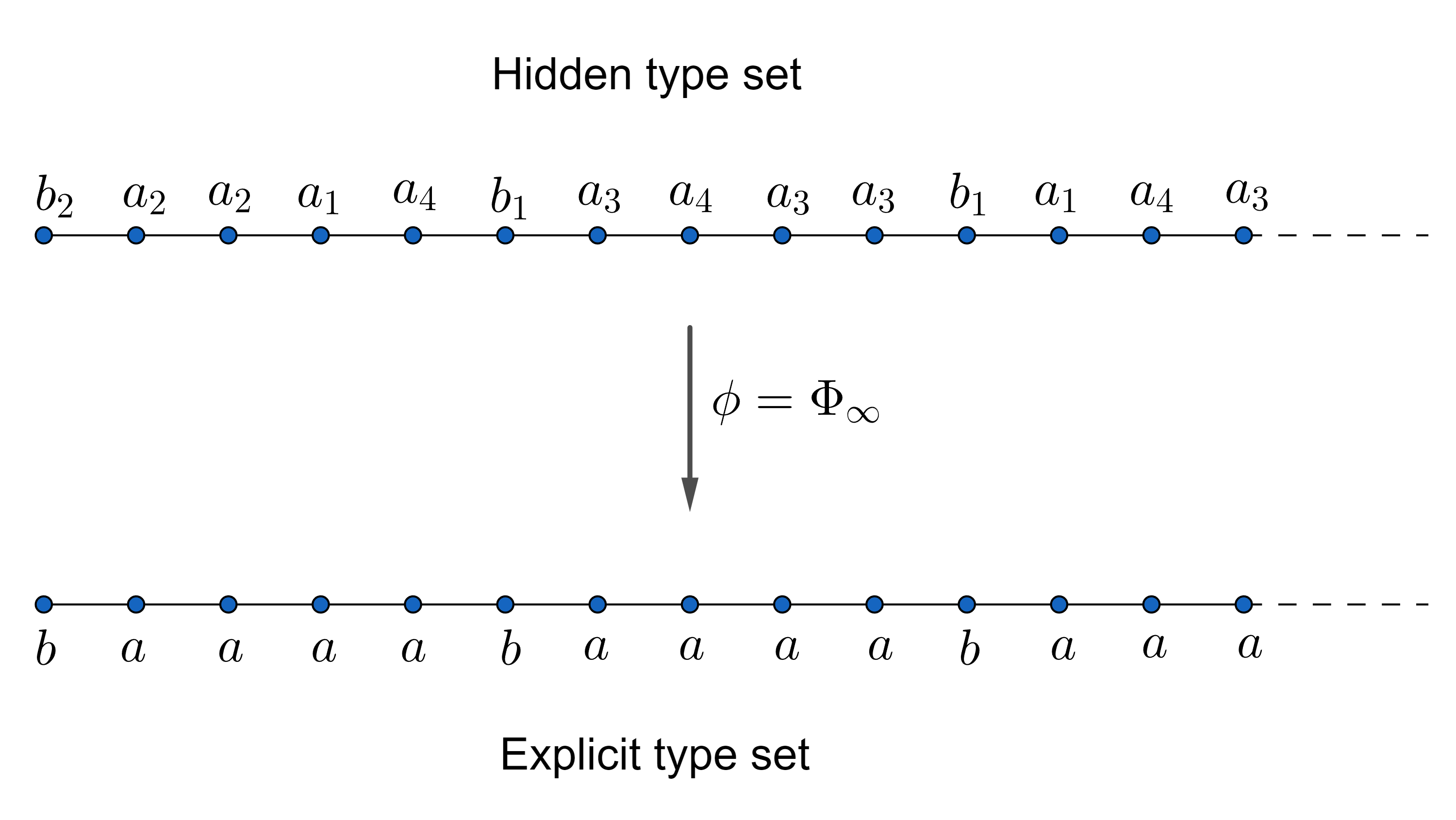}
    \caption{The graph of $0$-block code $\Phi:\{a_1,a_2,a_3,a_4,b_1,b_2\}\to \{a,b\}$ defined by $\Phi(a_i)=a$ and $\Phi(b_i)=b$ in $\mathbb{N}.$ }
    \label{fig: H-E 1-dim}
\end{figure}

\begin{figure}[H]
    \centering
    \includegraphics[width=0.8\linewidth]{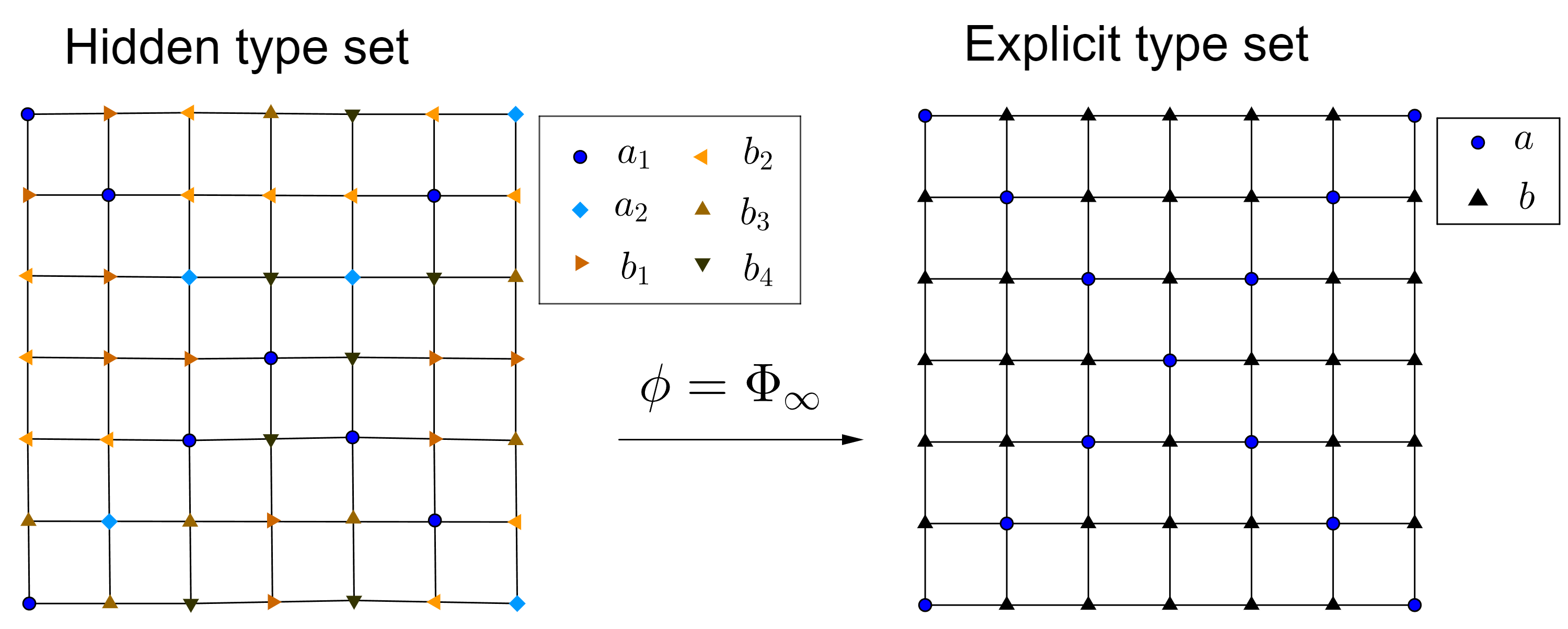}
    \caption{The graph of $0$-block code $\Phi:\{a_1,a_2,a_3,a_4,b_1,b_2\}\to \{a,b\}$ defined by $\Phi(a_i)=a$ and $\Phi(b_i)=b$ in $\mathbb{N}^2.$}
    \label{fig: H-E 2-dim}
\end{figure}

\section{Projected spread models}\label{sec: top}

We adopt the notation used in \cite{ban2023randomspread} throughout this paper. However, for the sake of readability and completeness, we have also summarized the notations in Table \ref{tab:symbol integration}.

Let the explicit type set be denoted by $\mathcal{B}=\{b_{i}\}_{i=1}^{K}$, and $\mathcal{T}_{d}$
be the conventional $d$-tree for $d\in \mathbb{N}$ with the root $\epsilon $%
. Define $\Sigma ^{s}=\{g\in \mathcal{T}_{d}:\left\vert g\right\vert =s\}$
for $s\in \mathbb{N}$ and $\Delta _{n}(h)=\{g\in \mathcal{T}_{d}:g$ is a
descendant of $h$ with $\left\vert g-h\right\vert \leq n\}$, where $%
\left\vert g-h\right\vert $ stands for the length of the unique path from $h$
to $g$ and $\left\vert g\right\vert =\left\vert g-\epsilon \right\vert $. Alternatively, when $h=\epsilon$, we express it as $\Delta _{n}(\epsilon )=\Delta _{n}=\cup _{i=0}^{n-1}\Sigma ^{i}
.$ Denote $\Delta _{m}^{n}=\Delta _{n}\backslash \Delta
_{m}=\{g\in \mathcal{T}_{d}:m<\left\vert g\right\vert \leq n\}$, and for $%
F\subseteq \mathcal{T}_{d}$ we define $F_{m}^{n}=F\cap \Delta _{m}^{n}$. An  $1$\emph{-pattern} is a function $p:F\rightarrow \mathcal{B}$ and $F=F_{p}$ is called the \emph{support }of $p,$ where $F\subseteq \Delta _{1}$ with $\epsilon\in F.$ Let 
$\mathcal{P}_{1}$ be the collection of all $1$-patterns, and for $p\in 
\mathcal{P}_{1}$, we write $p^{(0)}=p(\epsilon )\in \mathcal{B}$ and for $%
g_{1},\ldots ,g_{d_{p}}\in F_{p}$ with $\left\vert g_i\right\vert =1,$  $%
d_{p}\in \mathbb{N}$ and $i=1,2,...,d_p$, we write $p^{(1)}=(p(g_{1}),\ldots ,p(g_{d_{p}}))$.
Therefore, the $1$-pattern $p$ can be expressed as follows: 
\begin{eqnarray*}
p &=&(p^{(0)},p^{(1)}) \\
&=&(p(\epsilon );p(g_{1}),\ldots ,p(g_{d_{p}^{(1)}})).
\end{eqnarray*}%

The set $\mathcal{S}=\{p_i\}_{i=1}^L\subseteq \mathcal{P}_1$ is
called a \emph{spread model} if $\forall~p\in \mathcal{S}$ and $\forall~g\in
F_{p}$ with $\left\vert g\right\vert =1$, there exists a unique $q\in 
\mathcal{S}$ such that $q^{(0)}=p(g)$. Let $%
d=\max_{p\in \mathcal{S}}d_{p}$. The $d$-tree $\mathcal{T}_{d}$ mentioned in the previous paragraph is specified, along with other notations such as $\Delta{m}^{n}$, $F_{m}^{n}$, and so on. 

For an $1$-spread model $\mathcal{S}$ and $p\in \mathcal{S}$, we define $\tau _{p}^{\infty }$ as follows. If we let $\tau _{p}^{0}=p^{(0)}$ and $\tau
_{p}^{1}=p$, for $g\in F_{p}$ with $\left\vert g\right\vert =1$, since $%
\mathcal{S}$ is an $1$-spread model, there exists a unique $q_{g}\in \mathcal{S}$ with $q_{g}^{(0)}=q_{g}(\epsilon )=p(g)$. As a result, we substitute $p(g)$ with the $1$-pattern $q_{g}$ for all $g\in F_{p}$ where $\left\vert g\right\vert =1$, in order to generate a pattern $\tau_{p}^{2}$. Once $\tau_{p}^{n}$ is constructed, we replace the pattern $q_{g}$ with the symbol $\tau_ {p}^{n}(g)$ for all $g\in F_{\tau_{p}^{n}}$ where $\left\vert g\right\vert =n$, to generate $\tau_{p}^{n+1}$. Lastly, we define $\tau _{p}=\tau
_{p}^{\infty }=\lim_{n\rightarrow \infty }\tau _{p}^{n}$ and call it the 
\emph{infinite spread pattern induced from }$p$\emph{\ with respect to }$%
\mathcal{S}$ (\emph{induced spread pattern from} $p$).

Given $\tau _{p}$ for some $p\in \mathcal{S}$ and $p^{(0)}=p(\epsilon )=b\in 
\mathcal{B}$, suppose $F\subset F_{\tau _{p}}$ is a finite set, and we denote by 
$\tau _{p}|_{F}$ the subpattern of $\tau _{p}$ along the subset $F$, that
is, $\tau _{p}|_{F}=\{\tau _{p}(g):g\in F\}$. Given a sequence $%
\{k_{n}\}_{n=1}^{\infty }\subseteq \mathbb{N}$, the following value $s_{p}(b;%
\mathcal{S},\{k_{n}\}_{n=1}^{\infty })$ is of interest and significance for
the spread model $\mathcal{S}$. 
\[
s_{p}(a;\mathcal{S},\{k_{n}\}_{n=1}^{\infty })=\lim_{n\rightarrow \infty }%
\frac{O_{a}(\tau _{p}|_{\Delta _{s_{n}}^{s_{n+1}}(F_{\tau _{p}})})}{%
\left\vert \Delta _{s_{n}}^{s_{n+1}}(F_{\tau _{p}})\right\vert }\text{, }%
a\in \mathcal{B}\text{,}
\]%
where 
\[
s_{n}=\sum_{i=1}^{n}k_{i}\text{,}
\]%
and $O_{a}(\tau _{p}|_{F})$ is the \emph{number of occurrences of the type }$%
a$\emph{\ in the range }$F$.

Let $\mathcal{A}$ be another type set and $\Phi :\mathcal{B}\rightarrow 
\mathcal{A}$ be an assignment from $\mathcal{B}$ into $\mathcal{A}$. Given $%
\tau _{p}$ and its support $F_{\tau _{p}}$, we define 
\[
\phi (\tau _{p})=\left( \Phi (\tau _{p}(g))\right) _{g\in F_{\tau _{p}}}\in 
\mathcal{A}^{F_{\tau _{p}}}\text{.}
\]%
Here we call $\Phi $ a $0$\emph{-block code}, and the pair $(\mathcal{S}%
,\Phi )$ is called the \emph{projected spread model}. Given a projected spread model 
$(\mathcal{S},\Phi )$, for $b\in \mathcal{B}$, $a\in \mathcal{A}$ and $p\in 
\mathcal{S}$ with $p(\epsilon )=b$, the \emph{spread rate with respect to }$%
(S,\Phi )$ is defined as
\begin{equation}
s_{p}(a;\mathcal{S},\Phi ,\{k_{n}\}_{n=1}^{\infty })=\lim_{n\rightarrow
\infty }\frac{O_{a}(\phi (\tau _{p})|_{\Delta _{s_{n}}^{s_{n+1}}(F_{\phi (\tau
_{p})})})}{\left\vert \Delta _{s_{n}}^{s_{n+1}}(F_{\phi (\tau _{p})})\right\vert }%
\text{, }a\in \mathcal{A}\text{.}  \label{1}
\end{equation}%
The aim of this article is to provide a complete formula of the spread rate for a projected spread model.

\begin{table}[H]
    \centering 
    \begin{tabular}{ccl}
        \hline
        $\mathcal{B}$& & hidden type set\\
        $\mathbf{B}^{[k]}$& & hidden type set for $k$-pattern\\
        $\mathcal{A}$& & explicit type set\\
        $\mathcal{T}_d$& &  conventional $d$-tree \\
        $\Sigma^s$& & the node in $\mathcal{T}_d$ whose distance to the root is $s$.\\
        $\Delta_n(h)$& & the node in $\mathcal{T}_d$ whose distance to node $h$ is less than $n$.\\
        $\Delta_n$& & the node in $\mathcal{T}_d$ whose distance to the root is less than $n$.\\
        $\Delta_m^n$& & the node in $\mathcal{T}_d$ whose distance to the root is in $(m,n]$.\\
        $p$& & 1-pattern\\
        $F_p$& & support of $p$\\
        $\mathcal{P}_k$& &the collection of all k-patterns\\
        $\mathcal{S}$& & spread model\\
        $\mathbf{S}^{[k]}$& & $k$-spread model from $\mathcal{S}.$\\
        $\tau_p^\infty$& &infinite spread pattern induced from $p$ with respect to $\mathcal{S}$\\
        $O_{a}(\tau _{p}|_{F})$& &the number of occurrences of the type $a$ in the range $F$\\
        $\Phi$& & 0-block code\\
        $\Phi^{[k]}$& & $k$-block code\\
        $\theta_b$& & the set of all $(m-1)$-pattern $\overline{b}$ with $\overline{b}(\epsilon)=b$ \\
        $(\mathcal{S},\Phi)$& &projected spread model\\
        
        \hline
    \end{tabular}
    \caption{Frequently used notation}
    \label{tab:symbol integration}
\end{table}

\subsection{$1$-spread model and $0$-block code}

In this subsection, we present the spread rate for the projected spread model $(\mathcal{S},\Phi),$ where $\mathcal{S}$ is a 1-spread model and $\Phi$ is a 0-block code corresponding to the previous result in \cite{ban2023spread}, as stated below.

\begin{theorem}[Proposition 2 \cite{ban2023spread}]
Let $\mathcal{S=}\{p_{i}\}_{i=1}^{L}$ be a spread model over $\mathcal{B}
$, and write $\mathcal{B}=\{b_{j}\}_{j=1}^{K}$. Suppose $M$ is the
associated $\xi $-matrix, and $\rho =\rho _{M}$ is the maximal eigenvalue of $M$
with positive left eigenvector $w=(w(b_{j}))_{j=1}^{K}$ and $\sum_{j=1}^K \omega(b_j)=1$. Then for $b\in \mathcal{B%
}$, the vector $(s_{b}(b_{j},\mathcal{S},\{k_{n}\}_{n=1}^{\infty
}))_{j=1}^{K}$ is independent of $b$ and 
\[
(s_{b}(b_{j},\mathcal{S},\{k_{n}\}_{n=1}^{\infty
}))_{j=1}^{K}=w=(w(b_{j}))_{j=1}^{K}\text{.}
\]
\end{theorem}

The lemma used to prove the main result, which appears in \cite{ban2023mathematical}, is stated below for convenience of reference.
\begin{lemma}[Lemma 2 \cite{ban2023mathematical}]
\label{Lma: 2}Let $\{a_{n}\}_{n=1}^{\infty },$ $\{b_{n}\}_{n=1}^{\infty }$
be real sequences and $\{c_{n}\}_{n=1}^{\infty },$ $\{d_{n}\}_{n=1}^{\infty
} $ be positive real sequences. Suppose 
\[
\lim_{n\rightarrow \infty }\frac{a_{n}}{b_{n}}=\lim_{n\rightarrow \infty }%
\frac{c_{n}}{d_{n}}=Q\text{.} 
\]%
Then 
\[
\lim_{n\rightarrow \infty }\frac{a_{n}+c_{n}}{b_{n}+d_{n}}=Q\text{.} 
\]%
Furthermore, suppose that $\lim_{n\rightarrow \infty
}\sum_{j=1}^{n}b_{j}=\infty $, then 
\[
\lim_{n\rightarrow \infty }\frac{\sum_{j=1}^{n}a_{j}}{\sum_{j=1}^{n}b_{j}}=Q%
\text{.} 
\]
\end{lemma}

Let $(\mathcal{S},\Phi )$ be a projected spread model and $\Phi :\mathcal{B}%
\rightarrow \mathcal{A}$ be a $0$-block code. Let $b\in \mathcal{B}$ be a
certain type, let $\mathbf{1}_{b}\in \{0,1\}^{K}$ denote the vector with $1$ as its 
$b$th coordinate and $0$'s elsewhere, i.e., $\mathbf{1}%
_{b}=(0,\ldots ,0\overbrace{,1,}^{b\text{th}}0,\ldots ,0)^t$ and let $\mathbf{1\in 
}\{0,1\}^{K}$ denote the vector whose entries are all $1^{\prime }$s. Let $%
C\subseteq \mathcal{B}$ be a finite set of $\mathcal{B}$, we define $\mathbf{%
1}_{C}=\sum_{c\in C}\mathbf{1}_{c}$. For $a\in \mathcal{A}$, we define $\Phi
^{-1}(a)=\{c:\Phi (c)=a\}$ to be the preimage of $a$ with respect to $\Phi $.
For a projected spread model $\left( \mathcal{S},\Phi \right) $, we establish
the formula for the spread rate $s_{p}(a,\mathcal{S},\Phi
,\{k_{n}\}_{n=1}^{\infty })$ (cf. (\ref{1})) below.

\begin{theorem}
\label{Thm: 2}Let $\left( \mathcal{S},\Phi \right) $ be a projected spread model
with $\mathcal{S}$ being an $1$-spread model and $\Phi $ be an $0$-block
code, and $M$ be the $\xi $-matrix of the $1$-spread model $\mathcal{S}$.
Suppose $\{k_{n}\}_{n=1}^{\infty }\subseteq \mathbb{N}$ is a sequence such
that $k_{n}=k$ $\forall n\in \mathbb{N}$ or $k_{n}\rightarrow \infty $ as $%
n\rightarrow \infty $. For $b\in \mathcal{B}$, $a\in \mathcal{A}$, and $p\in 
\mathcal{S}$ with $p(\epsilon )=b$, then we have 
\begin{eqnarray*}
s_{p}(a,\mathcal{S},\Phi ,\{k_{n}\}_{n=1}^{\infty }) &=&\lim_{n\rightarrow
\infty }\frac{\mathbf{1}_{b}^tM^{n}\mathbf{1}_{\Phi ^{-1}(a)}}{\mathbf{1}%
_{b}^tM^{n}\mathbf{1}} \\
&=&\sum_{c\in \mathcal{B}:\text{ }\Phi (c)=a}w(c)\text{,}
\end{eqnarray*}%
where $w(c)$ is the $c$th component of the normalized left positive eigenvector $w$ of the $\xi $-matrix $M$ corresponding to the maximal eigenvalue $\rho _{M}$ of 
$M$.
\end{theorem}

\begin{proof}
Let $M$ be the $\xi $-matrix associated with the $1$-spread model $\mathcal{S}$%
. Assume that $k_{n}=1$ $\forall n\in \mathbb{N}$, and thus $\Delta
_{s_{n}}^{s_{n+1}}=\Delta _{n}^{n+1}$. For $b,c\in \mathcal{B}$,  $\tau
_{p}$ is the infinite spread pattern induced from $p$ with respect to $%
\mathcal{S}$. It is known that $\mathbf{1}_{b}^{t}M^{n}\mathbf{1}_{c}$ is
the number of $c^{\prime }$s in the support $F_{\tau _{p}}\cap \Sigma^{n}$, which indicates that 
\[
\mathbf{1}_{b}^{t}M^{n}\mathbf{1}_{c}=O_{c}(\tau _{p}|_{\Delta
_{s_{n}}^{s_{n+1}}(F_{\tau _{p}})})=O_{c}(\tau _{p}|_{\Delta _{n}^{n+1}(F_{\tau
_{p}})})=O_{c}(\tau _{p}|_{F_{\tau _{p}}\cap \Sigma^{n}})\text{.}
\]%
Since $\Phi (c)=a$, $\forall c\in \Phi ^{-1}(a)$, it indicates 
\[
\phi (\tau _{p})|_{\Delta _{n}^{n+1}(F_{\tau _{p}})}=\bigcup\limits_{c:\text{ }%
c\in \Phi ^{-1}(a)}\left\{ \Phi (\tau _{p}(g)):\tau _{p}(g)=c\text{ and }%
g\in \Delta _{n}^{n+1}(F_{\tau _{p}})\right\} 
\]%
Therefore, 
\begin{eqnarray}
O_{a}(\phi (\tau _{p})|_{\Delta _{s_{n}}^{s_{n+1}}(F_{\phi (\tau _{p})})})
&=&\sum_{c\in \mathcal{B}:\text{ }c\in \Phi ^{-1}(a)}\left\{ \Phi (\tau
_{p}(g)):\tau _{p}(g)=c\text{, }g\in \Delta _{n}^{n+1}(F_{\tau _{p}})\right\}  
\nonumber \\
&=&\sum_{c\in \mathcal{B}:\text{ }\Phi (c)=a}O_{c}(\tau _{p}|_{\Delta
_{n}^{n+1}(F_{\tau _{p}})})  \nonumber \\
&=&\mathbf{1}_{b}^{t}M^{n}\left( \sum_{c\in \mathcal{B}:\text{ }\Phi (c)=a}%
\mathbf{1}_{c}\right)   \nonumber \\
&=&\sum_{c\in \mathcal{B}:\text{ }\Phi (c)=a}\mathbf{1}_{b}^{t}M^{n}\mathbf{1%
}_{c}  \nonumber \\
&=&\mathbf{1}_{b}^{t}M^{n}\mathbf{1}_{\Phi ^{-1}(a)}.  \label{3}
\end{eqnarray}%
Combining (\ref{3}) with the fact that $\left\vert \Delta
_{s_{n}}^{s_{n+1}}(F_{\tau _{p}})\right\vert =\left\vert \Delta
_{n}^{n+1}(F_{\tau _{p}})\right\vert =\mathbf{1}_{b}^tM^{n}\mathbf{1}$ yields%
\begin{eqnarray*}
s_{p}(a,\mathcal{S},\Phi ,\{k_{n}\}_{n=1}^{\infty }) &=&\lim_{n\rightarrow
\infty }\frac{O_{a}(\phi (\tau _{p})|_{\Delta _{s_{n}}^{s_{n+1}}(F_{\tau
_{p}})})}{\left\vert \Delta _{s_{n}}^{s_{n+1}}(\tau _{p})\right\vert }
\\
&=&\lim_{n\rightarrow \infty }\frac{\sum_{c\in \mathcal{B}:\text{ }\Phi
(c)=a}\mathbf{1}_{b}^{t}M^{n}\mathbf{1}_{c}}{\mathbf{1}_{b}^tM^{n}\mathbf{1}}
\\
&=&\sum_{c\in \mathcal{B}:\text{ }\Phi (c)=a}\lim_{n\rightarrow \infty }%
\frac{\mathbf{1}_{b}^{t}M^{n}\mathbf{1}_{c}}{\mathbf{1}_{b}^tM^{n}\mathbf{1}}
\\
&=&\sum_{c\in \mathcal{B}:\text{ }\Phi (c)=a}w(c)
\end{eqnarray*}%
and it proves the result in the case where $k_{n}=1$ $\forall n\in \mathbb{N}$. For $k_{n}=k$
$\forall n\in \mathbb{N}$, we have $s_{n}=kn$ and thus 
\[
\left\vert \Delta _{s_{n}}^{s_{n+1}}(F_{\phi \left( \tau _{p}\right)}
)\right\vert =\left\vert \Delta _{s_{n}}^{s_{n+1}}(F_{\tau _{p}})\right\vert
=\sum_{i=1}^{k}\left\vert F_{\tau _{p}}\cap \Sigma^{kn+i}\right\vert 
\text{.}
\]%
The above argument demonstrates that for $1\leq j\leq k$, we have 
\begin{equation}
\lim_{n\rightarrow \infty }\frac{O_{a}(\phi (\tau _{p})|_{F_{\phi (\tau
_{p})}\cap \Sigma^{kn+j}})}{\left\vert F_{\phi (\tau _{p})}\cap \Sigma^{kn+j}\right\vert }=\sum_{c\in \mathcal{B}:\text{ }\Phi (c)=a}w(c)\text{%
.}  \label{2}
\end{equation}%
Hence, combining Lemma \ref{Lma: 2} with (\ref{2}) we have 
\begin{eqnarray*}
s_{p}(a,\mathcal{S},\Phi ,\{k_{n}\}_{n=1}^{\infty }) &=&\lim_{n\rightarrow
\infty }\frac{\sum_{i=1}^{k}O_{a}(\phi (\tau _{p})|_{F_{\phi (\tau
_{p})}\cap \Sigma^{kn+i}})}{\sum_{i=1}^{k}\left\vert F_{\phi (\tau
_{p})}\cap \Sigma^{kn+i}\right\vert } \\
&=&\sum_{c\in \mathcal{B}:\text{ }\Phi (c)=a}w(c)\text{.}
\end{eqnarray*}%
The proof of the case where $\{k_{n}\}_{n=1}^{\infty }$ with $%
k_{n}\rightarrow \infty $ as $n\rightarrow \infty $ is almost identical to the
above. Thus, we omit it and the proof is completed.
\end{proof}

\subsection{$1$-spread model and $k$-block code for $k\geq 1$}\label{sec: t_1skb}
Suppose $\mathcal{S}=\{p_{j}\}_{j=1}^{L}$ is an $1$-spread model over $%
\mathcal{B}.$  For all $p\in 
\mathcal{S}$ and $n\in \mathbb{N}$, since $\mathcal{S}$ is an $1$-spread
model, there exists exactly one $\tau _{p}^{n}$ in which $\tau _{p}^{1}=p$.
Fix $k\in \mathbb{N}$ and denote by $\mathcal{S}^{[k+1]}=\left\{\tau
_{p}^{k+1}\right\}_{p\in \mathcal{S}}$ a $(k+1)$-spread model from $\mathcal{S}.$ Let $\mathcal{P}_k=\left\{p:F\subseteq \Delta_{k}\to\mathcal{B}\right\}$ be the collection of all $k$-patterns. Denote by
\[
\mathbf{B}^{[k]}=\left\{\eta \in \mathcal{P}_{k}:\eta =\tau _{p}^{k}\text{ for
some }p\in \mathcal{S}\right\}
\]
and let $\Phi ^{\lbrack k]}:\mathcal{P}^{[k]}\rightarrow \mathcal{A}$
be a $k$-block code from $\mathcal{P}^{[k]}$ into $\mathcal{A}$. Let $b\in \mathcal{B}$, $a\in \mathcal{A}$
and $p\in \mathcal{S}$ with $p(\epsilon )=b$, and we define the \emph{spread
rate }with respect to $\{k_{n}\}_{n=1}^{\infty }$ as follows:  
\[
s_{p}(a,\mathcal{S},\Phi ^{\lbrack k]},\{k_{n}\}_{n=1}^{\infty
})=\lim_{n\rightarrow \infty }\frac{O_{a}(\phi (\tau _{p})|_{\Delta
_{s_{n}}^{s_{n+1}}(\tau _{p})})}{\left\vert \Delta _{s_{n}}^{s_{n+1}}(\tau
_{p})\right\vert }\text{.}
\]%

We note that if $p\in \mathcal{S}$, since $\mathcal{S}$ is an $1$-spread
model, there is only one way to extend $p$ to $\tau _{p}^{n}$ for $n\in 
\mathbb{N}$. Therefore, both sets $\{p\}_{p\in \mathcal{S}}$ and $\left\{\tau
_{p}^{k+1}\right\}_{p\in \mathcal{S}}$ have common cardinalities. Define 
\[
\mathbf{S}^{[k+1]}=\mathbf{S}^{[k+1]}(\mathcal{S}):=\{\mathbf{p}%
_{j}\}_{j=1}^{L}=\{\tau _{p}^{k+1}\}_{p\in \mathcal{S}}\text{,}
\]%
which is an ``$1$-spread model" over $\mathbf{B}^{[k]}$. We note that for $%
\mathbf{p}\in \mathbf{S}^{[k+1]}$, it can be viewed as an $1$-pattern with
respect to the type set $\mathbf{B}^{[k]}$, and it can also be viewed as
an $(k+1)$-pattern with respect to the type set $\mathcal{B}$. In this
circumstance, we use $\mathbf{\bar{p}}$ to denote the associated $(k+1)$%
-pattern. Suppose $M$ (resp. $\mathbf{M}^{[k+1]}$) is the $\xi $-matrix of $%
\mathcal{S}$ (resp. $\mathbf{S}^{[k+1]}$), and we have the following lemma.

\begin{lemma}\label{Lemma:xi-matrices }
\label{Lma: 1}$M=\mathbf{M}^{[k+1]}$
\end{lemma}

\begin{proof}
First we note that $M$ is indexed by $\mathcal{B}$ and $\mathbf{M}^{[k+1]}$
is indexed by $\mathbf{B}^{[k]}$. By the same argument as in the above
paragraph, we have $\left\vert \mathcal{B}\right\vert =\left\vert \mathbf{B}%
^{[k]}\right\vert $, where $\left\vert A\right\vert $ denotes the number of
the set $A$. Suppose $M(b,c)=1$, this means there exists a $p\in \mathcal{S}$
and $g\in F_{p}$ with $\left\vert g\right\vert =1$ such that $p(\epsilon )=b$
and $p(g)=c$. Since $\mathcal{S}$ is an $1$-spread model, there exists a
unique $1$-pattern, say $q\in \mathcal{S}$ such that $q(\epsilon )=c$. Let $%
\tau _{p}^{k}\in \mathbf{B}^{[k]}$ (resp. $\tau _{q}^{k}\in \mathbf{B}^{[k]}$%
) be the induced $k$-pattern from $p$ (resp. $q$), it can be easily seen
that $\tau _{p}^{k+1}\in \mathbf{S}^{[k+1]}$ and $\mathbf{M}^{[k+1]}(\tau
_{p}^{k},\tau _{p}^{k})=1$, thus $M=\mathbf{M}^{[k+1]}$. The proof is
completed.
\end{proof}

For $b\in \mathcal{B}$, we denote by $\mathbf{b}\in \mathbf{B}^{[k]}$ the
type in $\mathbf{B}^{[k]}$ if $\mathbf{\bar{b}}(\epsilon )=b$. The
following result gives the formula of the spread rate for the projected spread
model $\left( \mathcal{S},\Phi ^{\lbrack k]}\right) $ in which $\mathcal{S}$
is an $1$-spread model and $\Phi ^{\lbrack k]}:\mathcal{P}^{[k]}\rightarrow 
\mathcal{A}$ is a $k$-block code.

\begin{theorem}
\label{Thm: 4}Let $\left( \mathcal{S},\Phi ^{\lbrack k]}\right) $ be a projected
spread model with $\mathcal{S}=\{p_{j}\}_{j=1}^{L}$ is an $1$-spread model
and $\Phi ^{\lbrack k]}:\mathcal{P}^{[k]}\rightarrow \mathcal{A}$ be a $k$%
-block code, $k\geq 1$. Suppose $\mathbf{S}^{[k+1]}=\{\mathbf{p}%
_{j}\}_{j=1}^{L}$ is the $1$- spread model induced from $\mathcal{S}$ over $%
\mathbf{B}^{[k]}$ defined as above, and $\mathbf{\Phi }^{[k]}:\mathbf{B}^{[k]}\to \mathcal{A}$ is the
associated $0$-block code. If $%
k_{n}=k$ $\forall n\in \mathbb{N}$ or $k_{n}\rightarrow \infty $ as $%
n\rightarrow \infty $, then, for $b\in \mathcal{B}$, $a\in \mathcal{A}$, $%
p\in \mathcal{S}$ with $p(\epsilon )=b$, we have%
\begin{eqnarray*}
s_{p}(a,\mathcal{S},\Phi ^{\lbrack k]},\{k_{n}\}_{n=1}^{\infty }) &=&\sum_{%
\mathbf{c}:\text{ }\mathbf{\Phi }^{[k]}(\mathbf{c})=a}s_{\mathbf{p}}(\mathbf{%
c},\mathbf{S}^{[k+1]},\mathbf{\Phi }^{[k]},\{k_{n}\}_{n=1}^{\infty }) \\
&=&\sum_{\mathbf{c}:\text{ }\mathbf{\Phi }^{[k]}(\mathbf{c})=a}\mathbf{w}(%
\mathbf{c}) \\
&=&\sum_{c:\text{ }c\in \Gamma _{a}^{k}}w(c)\text{,}
\end{eqnarray*}%
where $\mathbf{w}$ (resp. $w$) is the normalized left eigenvector of the $\xi$-matrix $\mathbf{M}^{[k+1]}$ (resp. $M$) of $\mathbf{S}^{[k+1]}$ (resp. $S$) corresponding to the maximal eigenvalue $\rho_{\mathbf{M}}$ (resp. $\rho_M$) of $\mathbf{M}$ (resp. $M$),
$\Gamma _{a}^{k}=\{c\in \mathcal{B}:\Phi ^{\lbrack k]}(\tau _{q}^{k})=a\text{
and }~q\in\mathcal{S}~\text{with}~q(\epsilon )=c\}\text{,}$%
and $\mathbf{p}\in \mathbf{S}^{[k+1]}$ represents $p\in \mathcal{S}$.
\end{theorem}

\begin{proof}
\textbf{1}. For $b\in \mathcal{B}$ and $p\in \mathcal{S}$ with $p(\epsilon
)=b$. Suppose $\mathbf{b}\in \mathbf{B}^{[k]}$ is the unique type defined
as above, and $\mathbf{p}\in \mathbf{S}^{[k+1]}$ with $\mathbf{p}(\epsilon )=%
\mathbf{b}$. Note that if $a\in \mathcal{A}$, suppose $g\in F_{\tau _{p}}$
with $\left\vert g\right\vert =n$, if $\mathbf{\bar{c}}:=\tau _{p}|_{\left(
F_{\tau _{p}}\right) _{n}^{n+k-1}}\in \mathcal{P}_{k}$ and $\mathbf{\Phi }%
^{[k]}(\mathbf{\bar{c}})=a$ (recall that $\mathbf{c\in B}^{[k]}$ is a type
in $\mathbf{B}^{[k]}$ and $\mathbf{\bar{c}}$ is the corresponding $k$%
-pattern in $\mathcal{B}$). This indicates that 
\[
\Phi ^{\lbrack k]}\left( \tau _{p}|_{\left( F_{\tau _{p}}\right)
_{n}^{n+k-1}}\right) =\Phi ^{\lbrack k]}\left( \mathbf{\bar{c}}\right) =a
\]%
and $\mathbf{\Phi }^{[k]}(\tau _{\mathbf{p}}|_{F_{\mathbf{\tau }_{\mathbf{p}%
}}\cap \Sigma _{d}^{n}})=a$. Therefore, it follows from Theorem \ref{Thm: 2} that we have%
\begin{eqnarray*}
s_{p}(a,\mathcal{S},\Phi ^{\lbrack k]},\{k_{n}\}_{n=1}^{\infty })
&=&\lim_{n\rightarrow \infty }\frac{O_{a}(\phi (\tau _{p})|_{\Delta
_{s_{n}}^{s_{n+1}}(F_{\phi (\tau _{p})})})}{\left\vert \Delta
_{s_{n}}^{s_{n+1}}(F_{\phi (\tau _{p})})\right\vert } \\
&=&\sum_{\mathbf{\bar{c}}:\text{ }\Phi ^{\lbrack k]}(\mathbf{\bar{c}}%
)=a}\lim_{n\rightarrow \infty }\frac{O_{\mathbf{\bar{c}}}(\tau _{p}|_{\Delta
_{s_{n}}^{s_{n+1}+k-1}(F_{\tau _{p}})})}{\left\vert \Delta
_{s_{n}}^{s_{n+1}+k-1}(F_{\tau _{p}})\right\vert } \\
&=&\sum_{\mathbf{c}:\text{ }\mathbf{\Phi }^{[k]}(\mathbf{c}%
)=a}\lim_{n\rightarrow \infty }\frac{O_{\mathbf{c}}(\tau _{\mathbf{p}%
}|_{\Delta _{s_{n}}^{s_{n+1}}(F_{\tau _{\mathbf{p}}})})}{\left\vert \Delta
_{s_{n}}^{s_{n+1}}(F_{\tau _{\mathbf{p}}})\right\vert } \\
&=&\sum_{\mathbf{c}:\text{ }\mathbf{\Phi }^{[k]}(\mathbf{c})=a}s_{\mathbf{p}%
}(\mathbf{c},\mathbf{S}^{[k]},\mathbf{\Phi }^{[k]},\{k_{n}\}_{n=1}^{\infty })
\\
&=&\sum_{\mathbf{c}:\text{ }\mathbf{\Phi }^{[k]}(\mathbf{c})=a}\mathbf{w}(%
\mathbf{c})\text{.}
\end{eqnarray*}

\textbf{2}. For any $c\in \mathcal{B}$ we denote by $\mathbf{c}\in \mathbf{B}%
^{[k]}$ the unique type in $\mathbf{B}^{[k]}$ such that $\mathbf{\bar{c}}%
(\epsilon )=c$, and $q\in \mathcal{S}$ is such that $q(\epsilon )=c$. It
follows from the definitions of $\mathbf{B}^{[k]}$ and $\mathbf{\Phi }^{[k]}$ that we obtain 
\begin{eqnarray*}
\left\{ \mathbf{c}\in \mathbf{B}^{[k]}:\text{ }\mathbf{\Phi }^{[k]}(\mathbf{c%
})=a\right\}  &=&\left\{ \mathbf{\bar{c}}\in \mathcal{P}_{k}:\text{ }\Phi
^{\lbrack k]}(\mathbf{\bar{c}})=a\right\}  \\
&=&\left\{ \mathbf{\bar{c}}\in \mathcal{P}_{k}:\Phi ^{\lbrack k]}\left( \tau
_{q}^{k}\right) \text{ }=a\right\}. 
\end{eqnarray*}%
Under the same argument as Lemma \ref{Lma: 1}, $M$ and $\mathbf{M}^{[k+1]}$
have the same dimension and if $M$ is indexed by $\mathcal{B}=\{b_{1},\ldots
,b_{L}\}$, then $\mathbf{M}^{[k+1]}$ is indexed by 
\[
\mathbf{B}^{[k]}=\{\mathbf{b}_{1},\ldots ,\mathbf{b}_{L}\}=\{\tau
_{p_{1}}^{k},\ldots ,\tau _{p_{L}}^{k}\},
\]%
where $p_{i}(\epsilon )=b_{i}$ for $1\leq i\leq L$. Thus, Lemma \ref{Lma: 1}
is applied again, and we have 
\begin{eqnarray*}
\sum_{\mathbf{c\in B}^{[k]}:\text{ }\mathbf{\Phi }^{[k]}(\mathbf{c})=a}%
\mathbf{w}(\mathbf{c}) &=&\sum_{\mathbf{\bar{c}}\in \mathcal{P}_{k}:\text{ }%
\Phi ^{\lbrack k]}\left( \mathbf{\bar{c}}\right) \text{ }=a}\mathbf{w}(%
\mathbf{c}) \\
&=&\sum_{\mathbf{\bar{c}}\in \mathcal{P}_{k}:\text{ }\Phi ^{\lbrack
k]}\left( \tau _{q}^{k}\right) \text{ }=a}\mathbf{w}(\mathbf{c}) \\
&=&\sum_{c\in \mathcal{B}:\text{ }\Phi ^{\lbrack k]}\left( \tau
_{q}^{k}\right) \text{ }=a,\text{ }q(\epsilon )=c}w(c) \\
&=&\sum_{c\in \Gamma _{a}^{k}}w(c)\text{.}
\end{eqnarray*}%
This completes the proof.
\end{proof}

\subsection{$m$-spread model and $0$-block code for $m\geq 1$\textbf{\ }}

Let $\mathcal{S}=\{p_{j}\}_{j=1}^{L}$ be an $m$-spread model over $\mathcal{B%
}$, that is, $p_{j}\in \mathcal{P}_{m}$ $\forall~1\leq j\leq L$ and for any $%
p\in \mathcal{S}$ and $\forall g\in F_{p}$ with $\left\vert g\right\vert =1$%
, there exists a unique $q\in \mathcal{S}$ such that 
\begin{equation}
q|_{\left( F_{q}\right) _{0}^{m-1}}=p|_{\left( F_{p}\right) _{0}^{m-1}(g)}%
\text{.}  \label{7}
\end{equation}%
Suppose $\Phi :\mathcal{B}\rightarrow \mathcal{A}$ is a $0$-block code and $(%
\mathcal{S},\Phi )$ is the associated projected spread model. Denote by $\mathbf{%
B}^{[m-1]}=\mathbf{B}^{[m-1]}(\mathcal{S})$ the set of $\mathbf{b}\in 
\mathcal{P}_{m-1}$ over $\mathcal{B}$ in which $\mathbf{b}=p|_{\left(
F_{p}\right) _{0}^{m-1}(\epsilon )}$ or $\mathbf{b}=p|_{\left( F_{p}\right)
_{0}^{m-1}(g)}$, for some $p\in \mathcal{S}$ and $g\in F_{p}$ with $%
\left\vert g\right\vert =1$. That is, we collect all $\left( m-1\right) $%
-patterns that appear in $\left( F_{p}\right) _{0}^{m-1}(\epsilon )$ or $\left(
F_{p}\right) _{0}^{m-1}(g)$ of $p\in \mathcal{S}$ with $\left\vert
g\right\vert =1$. Define $\mathbf{S}^{[m]}=\mathbf{S}^{[m]}(\mathcal{S})$ as the collection of $1$-patterns over $\mathbf{B}^{[m-1]}$ as follows. For $%
p\in \mathcal{S}$ with $\mathbf{p}(\epsilon ):=p|_{\left( F_{p}\right)
_{0}^{m-1}(\epsilon )}$, and $\mathbf{p}(g_{1}):=p|_{\left( F_{p}\right)
_{0}^{m-1}(g_{1})},\ldots ,\mathbf{p}(g_{d_{p}}):=p|_{\left( F_{p}\right)
_{0}^{m-1}(g_{d_{p}})}\in \mathbf{B}^{[m-1]}$ in which $\left\vert
g_{i}\right\vert =1$ $\forall i=1,\ldots ,d_{p}$, we write 
\[
\mathbf{p}=(\mathbf{p}^{(0)};\mathbf{p}^{(1)})=(\mathbf{p}(\epsilon );%
\mathbf{p}(g_{1}),\ldots ,\mathbf{p}(g_{d_{p}})),
\]%
which is a $1$-pattern over $\mathbf{B}^{[m-1]}$, and we say that $\mathbf{p}\in 
\mathbf{S}^{[m]}$ \emph{represents} $p\in \mathcal{S}$ with respect to the
new type set $\mathbf{B}^{[m-1]}$. Finally, we define 
\[
\mathbf{S}^{[m]}=\{\mathbf{p}=(\mathbf{p}(\epsilon );\mathbf{p}%
(g_{1}),\ldots ,\mathbf{p}(g_{d_{p}})):\mathbf{p}\text{ represents }p\text{
with respect to }\mathbf{B}^{[m-1]}\}\text{\textbf{.}}
\]%

Due to the fact that $\mathcal{S}$ is an $m$-spread model over $\mathcal{B}$%
, it can be easily seen that $\mathbf{S}^{[m]}$ is a $1$-spread model over $%
\mathbf{B}^{[m-1]}$ (cf. (\ref{7})).

Let $b\in \mathcal{B}$ and 
\[
\theta _{b}=\{\mathbf{b}\in \mathbf{B}^{[m-1]}:\mathbf{\bar{b}}(\epsilon
)=b\}.
\]%
Let $\mathbf{p\in S}^{[m]}$ such that $\mathbf{p}^{(0)}=\mathbf{b}$ and $%
\tau _{\mathbf{p}}$ be the associated induced pattern from $\mathbf{p}$. For 
$b,c\in \mathcal{B}$ and $p\in \mathcal{S}$ with $p(\epsilon )=b$, Theorem %
\ref{Thm: 1} below provides a method to determine the value of 
\[
s_{p}(c,\mathcal{S},\{k_{n}\}_{n=1}^{\infty })=\lim_{n\rightarrow \infty }%
\frac{O_{c}(\tau _{p}|_{\Delta _{s_{n}}^{s_{n+1}}(F_{\tau _{p}})})}{\left\vert
\Delta _{s_{n}}^{s_{n+1}}(F_{\tau _{p}})\right\vert }\text{.}
\]

\begin{theorem}[Spread rate for $m$-spread model]
\label{Thm: 1}Let $\mathcal{S}=\{p_{i}\}_{i=1}^{L}$ be a $m$-spread model over $\mathcal{B}$, and $\mathbf{S}^{[m]}=\{\mathbf{p}_{i}\}_{i=1}^{\mathbf{L%
}}$ be the associated induced $1$-spread model over $\mathbf{B}^{[m-1]}$.
Suppose $k_{n}=k$ $\forall n\in \mathbb{N}$, or $k_{n}\rightarrow \infty $
as $n\rightarrow \infty $. Then for $b,c\in \mathcal{B}$ and $p\in \mathcal{S%
}$ with $p(\epsilon )=b$, we have 
\[
s_{p}(c,\mathcal{S},\{k_{n}\}_{n=1}^{\infty })=\lim_{n\rightarrow \infty }%
\frac{O_{c}(\tau _{p}|_{\Delta _{s_{n}}^{s_{n+1}}(F_{\tau _{p}})})}{\left\vert
\Delta _{s_{n}}^{s_{n+1}}(F_{\tau _{p}})\right\vert }=\sum_{\mathbf{c}:\text{ }%
\mathbf{c}\in \theta _{c}}\mathbf{w}(\mathbf{c})\text{,} 
\]%
where $\mathbf{w}(\mathbf{c})$ is the $\mathbf{c}$th component of the normalized left eigenvector $\mathbf{w}$ of the $\mathbf{\xi }$-matrix $\mathbf{M}=%
\mathbf{M}_{\mathbf{\xi }}$.
\end{theorem}

\begin{theorem}
\label{Thm: 3}Let $\mathcal{S}=\{p_{j}\}_{j=1}^{L}$ be an $m$-spread model
over $\mathcal{B}$ and $\Phi :\mathcal{B}\rightarrow \mathcal{A}$ be a $0$%
-block code. Suppose $(\mathcal{S},\Phi )$ is the projected spread model and $%
k_{n}=k$ $\forall n\in \mathbb{N}$, or $k_{n}\rightarrow \infty $ as $%
n\rightarrow \infty $. Then, for $b\in \mathcal{B}$, $a\in \mathcal{A}$ and $p\in 
\mathcal{S}$ with $p(\epsilon )=b$, we have 
\[
s_{p}(a,\mathcal{S},\Phi ,\{k_{n}\}_{n=1}^{\infty })=\sum_{c:\text{ }c\in
\Phi ^{-1}(a)}\sum_{\mathbf{c}:\text{ }\mathbf{c}\in \theta _{c}}\mathbf{w}(%
\mathbf{c})\text{,} 
\]%
where $\mathbf{w}(\mathbf{c})$ is the $\mathbf{c}$th component of the normalized left eigenvector $\mathbf{w}$ of the $%
\mathbf{\xi }$-matrix $\mathbf{M}=\mathbf{M}_{\mathbf{\xi }}$ from the $1$%
-spread model $\mathbf{S}^{[m]}$ over $\mathbf{B}^{[m-1]}$.
\end{theorem}

\begin{proof}
Suppose $\mathbf{B}^{[m-1]}=\mathbf{B}^{[m-1]}(\mathcal{S})$ and $\mathbf{S}%
^{[m]}=\mathbf{S}^{[m]}(\mathcal{S})$ are defined as above. Let $b\in 
\mathcal{B}$ and $p\in \mathcal{S}$ with $p(\epsilon )=b$. Since $\Phi $ is
a $0$-block code, we have 
\begin{equation}
\Delta _{s_{n}}^{s_{n+1}}(\phi \left( \tau _{p}\right) )=\Delta
_{s_{n}}^{s_{n+1}}(\tau _{p})  \label{5}
\end{equation}%
Moreover, given the range of $\Delta _{s_{n}}^{s_{n+1}}(F_{\tau _{p}})$, if $%
g\in \Delta _{s_{n}}^{s_{n+1}}(F_{\tau _{p}})$ with $\tau _{p}(g)=c$ for $c\in
\Phi ^{-1}(a)$, then we have $\phi (\tau _{p})(g)=a$. Thus, we have 
\begin{equation}
O_{a}(\phi (\tau _{p})|_{\Delta _{s_{n}}^{s_{n+1}}(F_{\phi (\tau
_{p})})})=\sum_{c:\text{ }\Phi ^{-1}(a)}O_{c}(\tau _{p}|_{\Delta
_{s_{n}}^{s_{n+1}}(F_{\tau_{p}})})\text{.}  \label{4}
\end{equation}%
Combining (\ref{4}), (\ref{5}) and Theorem \ref{Thm: 1}, we obtain 
\begin{eqnarray*}
s_{p}(a,\mathcal{S},\Phi ,\{k_{n}\}_{n=1}^{\infty }) &=&\lim_{n\rightarrow
\infty }\frac{O_{a}(\phi (\tau _{p})|_{\Delta _{s_{n}}^{s_{n+1}}(F_{\phi (\tau
_{p})})})}{\left\vert \Delta _{s_{n}}^{s_{n+1}}(F_{\phi \left( \tau _{p}\right)}
)\right\vert } \\
&=&\lim_{n\rightarrow \infty }\frac{\sum_{c:\text{ }c\in \Phi
^{-1}(a)}O_{c}(\tau _{p}|_{\Delta _{s_{n}}^{s_{n+1}}(F_{\tau _{p}})})}{%
\left\vert \Delta _{s_{n}}^{s_{n+1}}(F_{\tau _{p}})\right\vert } \\
&=&\sum_{c:\text{ }c\in \Phi ^{-1}(a)}\lim_{n\rightarrow \infty }\frac{%
O_{c}(\tau _{p}|_{\Delta _{s_{n}}^{s_{n+1}}(F_{\tau _{p}})})}{\left\vert \Delta
_{s_{n}}^{s_{n+1}}(F_{\tau _{p}})\right\vert } \\
&=&\sum_{c:\text{ }c\in \Phi ^{-1}(a)}\sum_{\mathbf{c}:\text{ }\mathbf{c}\in
\theta _{c}}\mathbf{w}(\mathbf{c})\text{,}
\end{eqnarray*}%
where $\mathbf{w}(\mathbf{c})$ is the $\mathbf{c}$th component of the normalized left eigenvector $\mathbf{w}$ of the $%
\mathbf{\xi }$-matrix $\mathbf{M}=\mathbf{M}_{\mathbf{\xi }}$ from the $1$%
-spread model $\mathbf{S}^{[m]}$ over $\mathbf{B}^{[m-1]}$. This completes
the proof.
\end{proof}

\subsection{General cases: $m$-spread model and $k$-block code}

\subsubsection{The case where $m-1=k$}

Suppose $\mathcal{S}=\{p_{j}\}_{j=1}^{L}$ is the $\left( m-1\right) $-spread
model over $\mathcal{B}$ and $\Phi ^{\lbrack k]}$ is a $k$-block code with $%
m-1=k$. Using the method in Section 2.2, we denote by $\mathbf{B}^{[k]}=%
\mathbf{B}^{[k]}(\mathcal{S})$ a new type set obtained from $\mathcal{S}
$. Let $\mathbf{S}^{[m]}=\mathbf{S}^{[k+1]}=\mathbf{S}^{[k+1]}(\mathcal{S})$
be the $1$-spread model according to the new type set $\mathbf{B}^{[k]}$.
Let $\mathcal{A}$ be a type set and suppose $\mathbf{\Phi }^{[k]}:\mathbf{B%
}^{[k]}\rightarrow \mathcal{A}$ is the induced 0-block code on $\mathbf{B}%
^{[k]}$. Since $\mathbf{S}^{[k+1]}$ is a $1$-spread model and $\mathbf{\Phi 
}^{[k]}$ is a 0-block code over $\mathbf{B}^{[k]}$, Theorem \ref{Thm: 2}
is applied to compute the spread rate of the projected spread model $(\mathbf{S}%
^{[k+1]},\mathbf{\Phi }^{[k]})$ over $\mathbf{B}^{[k]}$ below.

\begin{theorem}\label{Thm: 2.4-1}
Let $b\in \mathcal{B}$, $a\in \mathcal{A}$ and $p\in S$ with $p(\epsilon )=b$%
. Then we have 
\[
s_{p}(a,\mathcal{S},\Phi ^{\lbrack k]},\{k_{n}\}_{n=1}^{\infty })=\sum_{%
\mathbf{c}:\text{ }\mathbf{\Phi }^{[k]}(\mathbf{c})=a}s_{\mathbf{p}}(\mathbf{%
c},\mathbf{S}^{[k+1]},\mathbf{\Phi} ^{\lbrack k]},\{k_{n}\}_{n=1}^{\infty }) 
\]
\end{theorem}

\subsubsection{The case where $m-1<k$}

Suppose $\mathcal{S}=\{p_{j}\}_{j=1}^{L}$ is the $m$-spread model and $\Phi
^{\lbrack k]}$ is a $k$-block code with $m-1<k$. Let 
\[
\mathbf{B}^{[k]}=\mathbf{B}^{[k]}(\mathcal{S})=\{\tau _{p}^{k}:p\in \mathcal{%
S}\}\text{,} 
\]%
and $\mathbf{S}^{[k+1]}=\mathbf{S}^{[k+1]}(\mathcal{S})$ be the $1$-spread
model induced from $\mathcal{S}$ over $\mathbf{B}^{[k]}$, or more precisely, 
\[
\mathbf{S}^{[k+1]}=\{\tau _{p}^{k+1}:p\in \mathcal{S}\}\text{.} 
\]%
Denote by $\mathbf{\Phi }^{[k]}$ the associated 0-block code on $\mathbf{B}%
^{[k]}$. Thus, Theorem \ref{Thm: 4} is applied to compute the spread rate of
the projected spread model $(\mathbf{S}^{[k+1]},\mathbf{\Phi }^{[k]})$ over $%
\mathbf{B}^{[k]}$. The result of spread rate is the same as Theorem \ref{Thm: 2.4-1}.

\subsubsection{The case where $m-1>k$}

Suppose $\mathcal{S}=\{p_{j}\}_{j=1}^{L}$ is the $m$-spread model and $\Phi
^{\lbrack k]}$ is a $k$-block code with $m-1>k$. Denote 
\[
\mathbf{B}^{[k]}=\mathbf{B}^{[k]}(\mathcal{S})=\{\eta \in \mathcal{P}_{k}:%
\text{ }\eta \text{ appears in some }p\in \mathcal{S}\}\text{.} 
\]%
Let $\mathbf{S}^{[m-k]}=\mathbf{S}^{[m-k]}(\mathcal{S})$ be the $\left(
m-k\right) $-spread model induced from $\mathcal{S}$ over $\mathbf{B}^{[k]}$
and $\mathbf{\Phi }^{[k]}\mathbf{:B}^{[k]}\rightarrow \mathcal{A}$ be the 0-block code over $\mathbf{B}^{[k]}$. Therefore, Theorem \ref{Thm: 3} can be
used to solve the spread rate of the projected spread model $(\mathbf{S}^{[m-k]},%
\mathbf{\Phi }^{[k]})$ over $\mathbf{B}^{[k]}$.

\begin{theorem}
Let $b\in \mathcal{B}$, $a\in \mathcal{A}$ and $p\in S$ with $p(\epsilon )=b$%
. Then we have 
\[
s_{p}(a,\mathcal{S},\Phi ^{\lbrack k]},\{k_{n}\}_{n=1}^{\infty })=\sum_{%
\mathbf{c}:\text{ }\mathbf{\Phi }^{[k]}(\mathbf{c})=a}s_{\mathbf{p}}(\mathbf{%
c},\mathbf{S}^{[m-k]},\mathbf{\Phi }^{[k]},\{k_{n}\}_{n=1}^{\infty }) 
\]
\end{theorem}

\section{Random projected spread models}\label{random models}

In this section, we will introduce the projected random spread model using the branching processes. First of all, we consider a population that starts with one individual and consists of individuals of $K$ different types, say $b_1, b_2, \cdots, b_K$ and the set $\mathcal{B}=\{b_1, b_2, \cdots, b_K\}$ is called the type set. Let 
$$
\begin{array}{c}
\mathbf{Z}_n=(Z_{n,1}, Z_{n,2}, \cdots,Z_{n,K})
\end{array}
$$
be the population vector in the $n$th generation, where $Z_{n,i}$ is the number of individuals of type $b_i$ in the $n$th generation, $i=1,2,\cdots, K$. Assume that each individual in the population lives for a unit of time and, upon their death, produces their offspring independent of others in the same generation and in the past of the population. Assume that the production mechanism of each individual follows the probability distribution $\{p^{(b_i)}(\cdot)\}_{i=1}^{K}$, where $p^{(b_i)}(j_1, j_2, \cdots, j_K)$ is the probability that an individual of type $b_i$ produces $j_1$ children of type $b_1$, $j_2$ children of type $b_2$, $\cdots$, and $j_K$ children of type $b_K$. Then the process $\{\mathbf{Z}_n\}_{n\geq 0}$ is called a \emph{$K$-type branching process with offspring distribution} $\{p^{(b_i)}(\cdot)\}_{i=1}^{K}$. When we use such a process to model the spread of certain objects, we also call $\{\mathbf{Z}_n\}_{n\geq 0}$ a \emph{random $1$-spread model with spread distribution $\{p^{(b_i)}(\cdot)\}_{i=1}^{K}$}. A branching process in which each individual only has exactly one offspring is called a singular branching process. To avoid any trivial cases, throughout this paper, we assume the non-singularity for all the random spread models we work with. We refer readers to \cite{ban2021mathematical} for the settings of the random $1$-spread models using branching processes. 

In a random $1$-spread model, we are interested in what happens to the type composition of the population in the long run. We define the spread rate of type $b_i$ when the spread is initiated by an individual of type $b_{i_0}$, $i=1,2,\cdots, K$ to be the limit of the proportion within the population as the following:
$$
\begin{array}{c}
s_{b_{i_0}}(b_j;\{\mathbf{Z}_n\}_{n\geq 0})\equiv \displaystyle{\lim_{n\rightarrow \infty} \frac{Z_{n,j}}{\sum\limits_{j=1}^KZ_{n,j}}}.
\end{array}
$$
Note that in a random $1$-spread model, the spread rate is also a random quantity. In the classical theory of branching processes, it is known that the behavior of the offspring matrix provides the information for the investigation of the branching process in the long run. So, the results about the $\xi$-matrices in the topological spread models would give some insights when we are dealing with the random spread models. Throughout this paper, we assume that the underlying probability space is $(\Omega, \mathcal{F}, P)$ and more details about the theory of branching processes can be found in \cite{Athreya2004}.

Let $m_{ij}=E(Z_{1,j}|\mathbf{Z}_0=\mathbf{1}_{b_i})$ be the expected value of the number of children of type $b_j$ produced by an individual of type $b_i$, where $\mathbf{1}_{b_i}$ is the unit vector with $1$ as its $i$th component. Then the matrix 
$$
M\equiv (m_{ij})=\left(
\begin{array}{cccc}
m_{11}& m_{12}&\cdots& m_{1K}\\
m_{21}& m_{22}&\cdots& m_{2K}\\
\vdots &\vdots &&\vdots \\
m_{K1}& m_{K2}&\cdots& m_{KK}
\end{array}
\right)$$
is called the \emph{offspring mean matrix} for this branching process $\{\mathbf{Z}_n\}_{n\geq 0}$. $M$ is also called the \emph{spread mean matrix} when $\{\mathbf{Z}_n\}_{n\geq 0}$ is considered as a random spread model.  Moreover, if $M^2=M\cdot M$ and $M^n=M^{n-1}\cdot M$ for all $n\geq 3$, then the $(i,j)$-entry $m^{(n)}_{ij}$ of the matrix $M^n$ is the expected value
$$
\begin{array}{c}
m^{(n)}_{ij}=E(Z_{n,j}|\mathbf{Z}_0=\mathbf{1}_{b_i})
\end{array}
$$
of the number of offspring of type $b_j$ in the $n$th generation of the population initiated by an ancestor of type $b_i$.

Let $\rho$ be the maximal eigenvalue of the offspring mean matrix $M$ of the branching process $\{\mathbf{Z}_n\}_{n\geq 0}$. We say $w=(w_1, w_2, \cdots, w_K)$ is the normalized left eigenvector of $M$ associated with $\rho$, if $wM=\rho w$ and $w_1+w_2+\cdots+w_K=1$.

The following theorem tells us that the spread rates are related to the components of this left eigenvector $w$.

\begin{proposition}[Spread rate for random $1$-spread model, \cite{ban2021mathematical}, \cite{ban2023randomspread}]\label{prop_randpm 1-spread model} Let $\{\mathbf{Z}_n\}_{n\geq 0}$ be a random $1$-spread model with the type set $\mathcal{B}=\{b_i\}_{i=1}^K$. Suppose that  $E(Z_{1,j}\log Z_{1,j}|\mathbf{Z}_0=\mathbf{1}_{b_i})<\infty$ for all $i,j=1,2,\cdots,K$ and that $\rho>1$ is the maximal eigenvalue of the spread mean matrix $M$ of $\{\mathbf{Z}_n\}_{n\geq 0}$ with positive normalized left eigenvector $w=(w_1, \cdots, w_K)$. Then
\begin{enumerate}
  \item [(i)] there exists a random variable $W$ such that 
  $$
  \begin{array}{c}
  \displaystyle{\lim_{n\rightarrow \infty}\frac{\mathbf{Z}_n}{\rho^n}}= w W \hspace{1cm} \textrm{ a.s.; and }
  \end{array}
  $$
  \item [(ii)] for any $b_{i_0}\in \mathcal{B}$, on the event $E_{i_0}$ of non-extinction given that $\mathbf{Z}_0=\mathbf{1}_{b_{i_0}}$, the spread rate is
$$
\begin{array}{c}
s_{b_{i_0}}(b_j;\{\mathbf{Z}_n\}_{n\geq 0})\equiv \displaystyle{\lim_{n\rightarrow\infty} \frac{Z_{n,j}}{\sum\limits_{j=1}^KZ_{n,j}}}=w_j \hspace{1cm} \textrm{ a.s., }
\end{array}
$$
which is independent of $b_{i_0}$.
\end{enumerate}
\end{proposition}

\begin{table}[H]
    \centering 
    \begin{tabular}{ccl}
        \hline
        $\mathbf{B}_{b_i}^{[k]}$& & collection of all potential $k$-patterns initiated with a root of\\
        &&hidden type $b_i\in \mathcal{B}$\\
        $\mathbf{B}^{[k]}$& & collection of all potential $k$-patterns\\
        $\{\mathbf{Z}_n\}_{n\geq 0}$& &  random $1$-spread model \\
        $\{\mathbf{Z}^{\Phi}_n\}_{n\geq 0}$& & the random projected spread model induced from $\{\mathbf{Z}_n\}_{n\geq 0}$\\
        &&and the $0$-block code $\Phi$\\
        $\mathcal{R}^{[k]}$& & the natural random $k$-spread model induced from $\{\mathbf{Z}_n\}_{n\geq 0}$\\
         $\{\mathbf{Z}^{[k]}_n\}_{n\geq 0}$& & the random $1$-spread model induced from $\{\mathbf{Z}_n\}_{n\geq 0}$ via the\\
        &&  natural random $k$-spread model $\mathcal{R}^{[k]}$\\
         $\{\mathbf{Z}^{[k],\Phi}_n\}_{n\geq 0}$& & the random projected spread model induced from $\{\mathbf{Z}_n\}_{n\geq 0}$\\
        &&and the $k$-block code $\Phi^{[k]}$\\
        \hline
    \end{tabular}
    \caption{Frequently used notation for random models}
    \label{tab:symbol integration for random}
\end{table}

\subsection{Random $1$-spread model and $0$-black code}\label{r_1s0b model}

Now, we consider a $0$-block code $\Phi: \mathcal{B}\rightarrow \mathcal{A}$ as defined in Section \ref{sec: top}, where $\mathcal{B}=\{b_1, b_2,\cdots,b_{K}\}$ is the hidden type set and  $\mathcal{A}=\{a_1, a_2,\cdots,a_{K'}\}$ is the explicit type set.

Now, for each $n=0, 1, 2, \cdots$, we let the random vector 
$$
\begin{array}{c}
\mathbf{Z}^{\Phi}_n=(Z^{\Phi}_{n,1}, Z^{\Phi}_{n,2}, \cdots,Z^{\Phi}_{n,K'})
\end{array}
$$
be the population vector in the $n$th generation categorized by type set $\mathcal{A}$. That is, $Z^{\Phi}_{n,j}$ is the number of individuals of explicit type $a_j$ (no matter what hidden types they are of) in the $n$th generation, $j=1,2,\cdots, K'$. Then, for each $j=1,2,\cdots, K'$, we have that
$$
\begin{array}{c}
Z^{\Phi}_{n,j}=\displaystyle{\sum_{i: \Phi(b_i)=a_j}Z_{n,i}}.
\end{array}
$$
where $n=0, 1, 2, \cdots$. Then 
the process $\{\mathbf{Z}^{\Phi}_n\}_{n\geq 0}$ is called \emph{ the random projected spread model induced from the random $1$-spread model $\{\mathbf{Z}_n\}_{n\geq 0}$ and the $0$-block code $\Phi: \mathcal{B}\rightarrow \mathcal{A}$}. In this random projected spread model initiated by an individual of hidden type $b_{i}$, we define the spread rate of the explicit type $a_j$ to be the following random quantity:
$$
\begin{array}{c}
s_{b_{i}}(a_j;\{\mathbf{Z}_n\}_{n\geq 0},\Phi)=\displaystyle{\lim_{n\rightarrow \infty}\frac{Z^{\Phi}_{n,j}}{\sum\limits_{j=1}^{K'} Z^{\Phi}_{n,j}}}
\end{array}
 $$
where $i=1,2,\cdots,K$ and  $j=1,2,\cdots,K'$. More generally, if given a sequence $
\{k_{n}\}_{n=1}^{\infty }\subseteq \mathbb{N}$ and $ s_{n}=\sum_{i=1}^{n}k_{i}$, we define 
$$
s_{b_i}(a_j;\{\mathbf{Z}_n\}_{n\geq 0},\Phi,\{k_{n}\}_{n=1}^{\infty })=\displaystyle{\lim_{n\rightarrow \infty }
\frac{\sum\limits_{r=s_n+1}^{s_{n+1}}Z^{\Phi}_{r,j}}{\sum\limits_{r=s_n+1}^{s_{n+1}}\sum\limits_{j=1}^{K'}Z^{\Phi}_{r,j}}},
$$
which, if it exists in some sense, gives the information of the limit proportion of individuals of explicit type $a_j$ within generations between the $(s_n+1)$th generation and the $s_{n+1}$th generation as $n\rightarrow \infty$.  Note that, if $k_n=1$ for all $n$, then 
$$
\begin{array}{c}
s_{b_i}(a_j;\{\mathbf{Z}_n\}_{n\geq 0},\Phi,\{k_{n}\}_{n=1}^{\infty })=s_{b_i}(a_j;\{\mathbf{Z}_n\}_{n\geq 0},\Phi).
\end{array}
$$

\begin{theorem}\label{thm_random 1s0b}
Let $\{\mathbf{Z}_n\}_{n\geq 0}$ be a random $1$-spread model with the type set $\mathcal{B}=\{b_i\}_{i=1}^K$, $\mathcal{A}=\{a_i\}_{i=1}^{K'}$ be another type set and $\Phi:\mathcal{B}\rightarrow\mathcal{A}$ be a $0$-block code. Suppose that  $E(Z_{1,j}\log Z_{1,j}|\mathbf{Z}_0=\mathbf{1}_{b_i})<\infty$ for all $i,j=1,2,\cdots,K$ and suppose that $\rho>1$ is the maximal eigenvalue of the spread mean matrix $M$ of $\{\mathbf{Z}_n\}_{n\geq 0}$ with positive normalized left eigenvector $w=(w_1, \cdots, w_K)$ and that $
\{k_{n}\}_{n=1}^{\infty }\subseteq \mathbb{N}$ is a sequence such that $k_n=k$ for all $n$ or $k_n\rightarrow \infty$ as $n\rightarrow \infty$. Then
\begin{enumerate}
  \item [(i)] let the random variable $W$ be as defined in Proposition \ref{prop_randpm 1-spread model}, then there exists a vector $u$ such that
  $$
  \begin{array}{c}
  \displaystyle{\lim_{n\rightarrow \infty}\frac{\mathbf{Z}^{\Phi}_n}{\rho^n}}= uW \hspace{1cm} \textrm{ a.s.; and }
  \end{array}
  $$
  \item [(ii)] for any $b_{i_0}\in \mathcal{B}$, on the event  of non-extinction given that $\mathbf{Z}_0=\mathbf{1}_{b_{i_0}}$, the spread rate is
$$
\begin{array}{c}
s_{b_{i_0}}(a_j;\{\mathbf{Z}_n\}_{n\geq 0},\Phi,\{k_{n}\}_{n=1}^{\infty })=\displaystyle{\sum_{i: \Phi(b_i)=a_j}w_i}\hspace{1cm} \textrm{ a.s. }
\end{array}
$$
and is independent of $b_{i_0}$.
\end{enumerate}
\end{theorem}

\begin{proof}
\begin{enumerate}
\item [(i)] Since 
$$
\begin{array}{c}
Z^{\Phi}_{n,j}=\displaystyle{\sum_{i: \Phi(b_i)=a_j}Z_{n,i}}
\end{array}
$$
and by Proposition \ref{prop_randpm 1-spread model} (i), we have that
 $$
  \begin{array}{rl}
  \displaystyle{\lim_{n\rightarrow \infty}\frac{Z^{\Phi}_{n,j}}{\rho^n}}&=\displaystyle{\lim_{n\rightarrow \infty}\frac{\sum\limits_{i: \Phi(b_i)=a_j}Z_{n,i}}{\rho^n}=\sum_{i: \Phi(b_i)=a_j}\lim_{n\rightarrow \infty}\frac{Z_{n,i}}{\rho^n}}\\
  &= \displaystyle{\sum_{i: \Phi(b_i)=a_j} w_i W} 
  \end{array}
  $$
almost surely. For each $j=1, 2, \cdots, K'$, let
$$
\begin{array}{c}
u_j=\sum\limits_{i: \Phi(b_i)=a_j} w_i
\end{array}
$$
and let $u=(u_1, u_2, \cdots, u_{K'})$, then for every $j=1, 2, \cdots, K'$, we have 
  $$
  \begin{array}{c}
  \displaystyle{\lim_{n\rightarrow \infty}\frac{Z^{\Phi}_{n,j}}{\rho^n}}= u_jW \hspace{1cm} \textrm{ a.s. }
  \end{array}
  $$
and hence
  $$
  \begin{array}{c}
  \displaystyle{\lim_{n\rightarrow \infty}\frac{\mathbf{Z}^{\Phi}_n}{\rho^n}}= uW \hspace{1cm} \textrm{ a.s. }
  \end{array}
  $$

\item [(ii)] First of all, let $E_{i_0}$ be the event of non-extinction given that $\mathbf{Z}_0=\mathbf{1}_{b_{i_0}}$ and, by Proposition \ref{prop_randpm 1-spread model} (ii), we have that
$$
\begin{array}{c}
 \displaystyle{\lim_{n\rightarrow\infty} \frac{Z_{n,i}}{\sum\limits_{j=1}^KZ_{n,j}}}=w_i \hspace{1cm} \textrm{ a.s. on $E_{i_0}$}
\end{array}
$$
for any $i=1,2,\cdots, K$. Now, for each $i$, let
$$
\begin{array}{c}
D_{i_o,i}=\bigg\{\omega \in E_{i_0}: \displaystyle{\lim_{n\rightarrow\infty} \frac{Z_{n,i}(\omega)}{\sum\limits_{j=1}^KZ_{n,j}(\omega)}}=w_i \bigg\}
\end{array}
$$
then we have that the probability of event $E_{i_0}\setminus D_{i_0,i}$ happening is zero and hence $E_{i_0}\setminus \cap_{i=1}^{K} D_{i_0,i}$ is also an event with probability zero. Also, for each $\omega \in \cap_{i=1}^{K} D_{i_0,i}$, we have that
$$
\begin{array}{c}
\displaystyle{\lim_{n\rightarrow\infty} \frac{Z_{n,i}(\omega)}{\sum\limits_{j=1}^KZ_{n,j}(\omega)}}=w_i \hspace{0.5cm} \textrm{ and }  \hspace{0.5cm}\displaystyle{\lim_{n\rightarrow \infty}\sum\limits_{j=1}^KZ_{n,j}(\omega)= \infty}
\end{array}
$$
for all $i=1,2,\cdots,K$. Hence, for every $\omega \in \cap_{i=1}^{K} D_{i_0,i}$
$$
\begin{array}{rl}
&s_{b_{i_0}}(a_j;\{\mathbf{Z}_n\}_{n\geq 0},\Phi)(\omega)\\
=&\displaystyle{\lim_{n\rightarrow \infty}\frac{Z^{\Phi}_{n,j}(\omega)}{\sum\limits_{j=1}^{K'} Z^{\Phi}_{n,j}(\omega)}=\lim_{n\rightarrow \infty}\frac{\sum\limits_{i: \Phi(b_i)=a_j}Z_{n,i}(\omega)}{\sum\limits_{j=1}^{K} Z_{n,j}(\omega)}}\\
=&\displaystyle{\lim_{n\rightarrow \infty}\sum\limits_{i: \Phi(b_i)=a_j}\frac{Z_{n,i}(\omega)}{\sum\limits_{j=1}^{K} Z_{n,j}(\omega)}=\sum\limits_{i: \Phi(b_i)=a_j}\lim_{n\rightarrow \infty}\frac{Z_{n,i}(\omega)}{\sum\limits_{j=1}^{K} Z_{n,j}(\omega)}}\\
=&\sum\limits_{i: \Phi(b_i)=a_j} w_i 
\end{array}
 $$
and therefore
$$
\begin{array}{c}
s_{b_{i_0}}(a_j;\{\mathbf{Z}_n\}_{n\geq 0},\Phi)=\displaystyle{\sum_{i: \Phi(b_i)=a_j}w_i}\hspace{1cm} \textrm{ a.s. on $E_{i_0}$}.
\end{array}
$$
 On the other hand, the assumption that the sequence $\{k_n\}$ increases to infinity gives that $k_{n+1}=s_{n+1}-s_n \rightarrow \infty$ as $n\rightarrow\infty$. So, it follows from the above results and Lemma \ref{Lma: 2} that
$$
\begin{array}{rl}
&s_{b_{i_0}}(a_j;\{\mathbf{Z}_n\}_{n\geq 0},\Phi,\{k_{n}\}_{n=1}^{\infty })=\displaystyle{\lim_{n\rightarrow \infty }
\frac{\sum\limits_{r=s_n+1}^{s_{n+1}}Z^{\Phi}_{r,j}}{\sum\limits_{r=s_n+1}^{s_{n+1}}\sum\limits_{j=1}^{K'}Z^{\Phi}_{r,j}}}\\
&=\displaystyle{\lim_{n\rightarrow \infty}\frac{Z^{\Phi}_{n,j}}{\sum\limits_{j=1}^{K'} Z^{\Phi}_{n,j}}=\sum\limits_{i: \Phi(b_i)=a_j} w_i}
\end{array}
$$
a.s. on $E_{i_0}$ and it also shows that the limit is independent of the initial type $b_{i_0}$. 
\end{enumerate}
\end{proof}

\subsection{Random $1$-spread model and $k$-block code for $k\geq 1$}\label{r_1skb model}

Since, for a random spread model, each realization (possible spread history structure) can be visualized as a rooted tree in which each node is labelled by a type, we can not only encode each type, i.e. use a $0$-block code to project every hidden type to an explicit type, but also encode each $k$-level label sub-tree using a so-called $k$-block code, as introduced in Section \ref{sec: top}. So, in this section, we will generalize the idea about the spread models projected by $k$-block codes from topological cases to random cases. However, in order to avoid confusion and complex notations, we will give the setting for $k\geq 1$ but introduce the detailed construction process with an example in the case of $k=1$. Similar steps and discussions in the $k=1$ case can be adopted for the general cases with $k\geq 1$.

First of all, we consider a random $1$-spread model $\{\mathbf{Z}_n\}_{n\geq 0}$ with the type set $\mathcal{B}=\{b_1, b_2, \cdots, b_K\}$, spread distribution $\{p^{(b_i)}(\cdot)\}_{i=1}^{K}$ and spread mean matrix $M$. Let $\mathcal{P}_k$ be the collection of all $k$-patterns and $\mathcal{A}=\{a_1, \cdots, a_{K'}\}$ be another type set as introduced in Section \ref{sec: top}. Recall that a $k$-block code $\Phi^{[k]}: \mathcal{P}_k\rightarrow \mathcal{A} $ is a map from $\mathcal{P}_k$ to $\mathcal{A}$. A $k$-block code $\Phi^{[k]}$ is applied to realizations of the random spread model in which every $k$-pattern appearing in the realizations is projected to an explicit type in $\mathcal{A}$ according to $\Phi^{[k]}$. That means that the hidden type of each node in the realizations is replaced by the image of the $k$-pattern rooted at that node under $\Phi^{[k]}$. Therefore, each realization will be transformed into a new rooted tree in which the type of each node is an explicit type in $\mathcal{A}$. In this case, our problem of interest is what happens to the proportion of each explicit type in the population in the long run.

In order to address the spread rates of explicit types after the spread model is projected, we define the potential patterns as follows:  For any positive integer $k$, a potential $k$-pattern with root of type $b_i$ is a $k$-pattern which is initiated with the type $b_i$ and has a positive probability of occurring. A potential $k$-pattern can represent a possible family structure in a branching process or a possible spread result in a random spread model, initiated by one individual, within $k$ generations. Note that whether or not a $k$-pattern is a potential $k$-pattern depends on the spread distribution of the random $1$-spread model.

\begin{example}\label{ex: r1}
If we consider the following random 1-spread model with type set $\mathcal{B}=\{b_1,b_2\}$ and spread distribution $\{p^{(b_i)}(\cdot)\}_{i=1}^{2}$:
$$
\begin{array}{c}
p^{(b_1)}(1,1)=\frac{1}{3},\hspace{0.3cm} p^{(b_1)}(2,1)=\frac{2}{3},\hspace{0.3cm} p^{(b_2)}(1,0)=\frac{1}{2},\hspace{0.3cm} p^{(b_2)}(1,1)=\frac{1}{2}
\end{array}
$$
then this model has four potential $1$-patterns listed in Figure \ref{Fig: potential 1-patterns}.
\begin{figure}[H] 
	\centering 
	\includegraphics[width=1\textwidth]{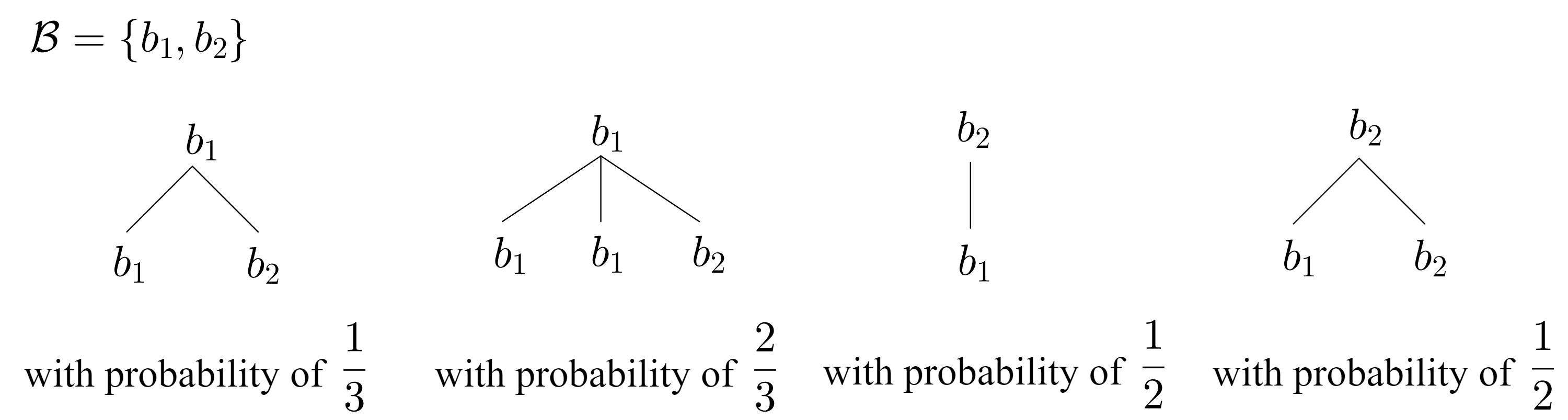} 
	\caption{The potential $1$-patterns} 
 \label{Fig: potential 1-patterns}
\end{figure}

In addition, the pattern in Figure \ref{Fig: potential 2-patterns} is a potential $2$-pattern for this spread model, since its probability of occurring is $\frac{1}{9}>0$, but the pattern in Figure \ref{Fig: non-potential 2-patterns} is not a potential $2$-pattern, for its probability of occurring is zero due to the fact that the event in which the sub-structure marked in the red square will happen is an impossible event when the spread follows the given spread distribution.

\begin{figure}[H] 
	\centering 
	\includegraphics[width=0.4\textwidth]{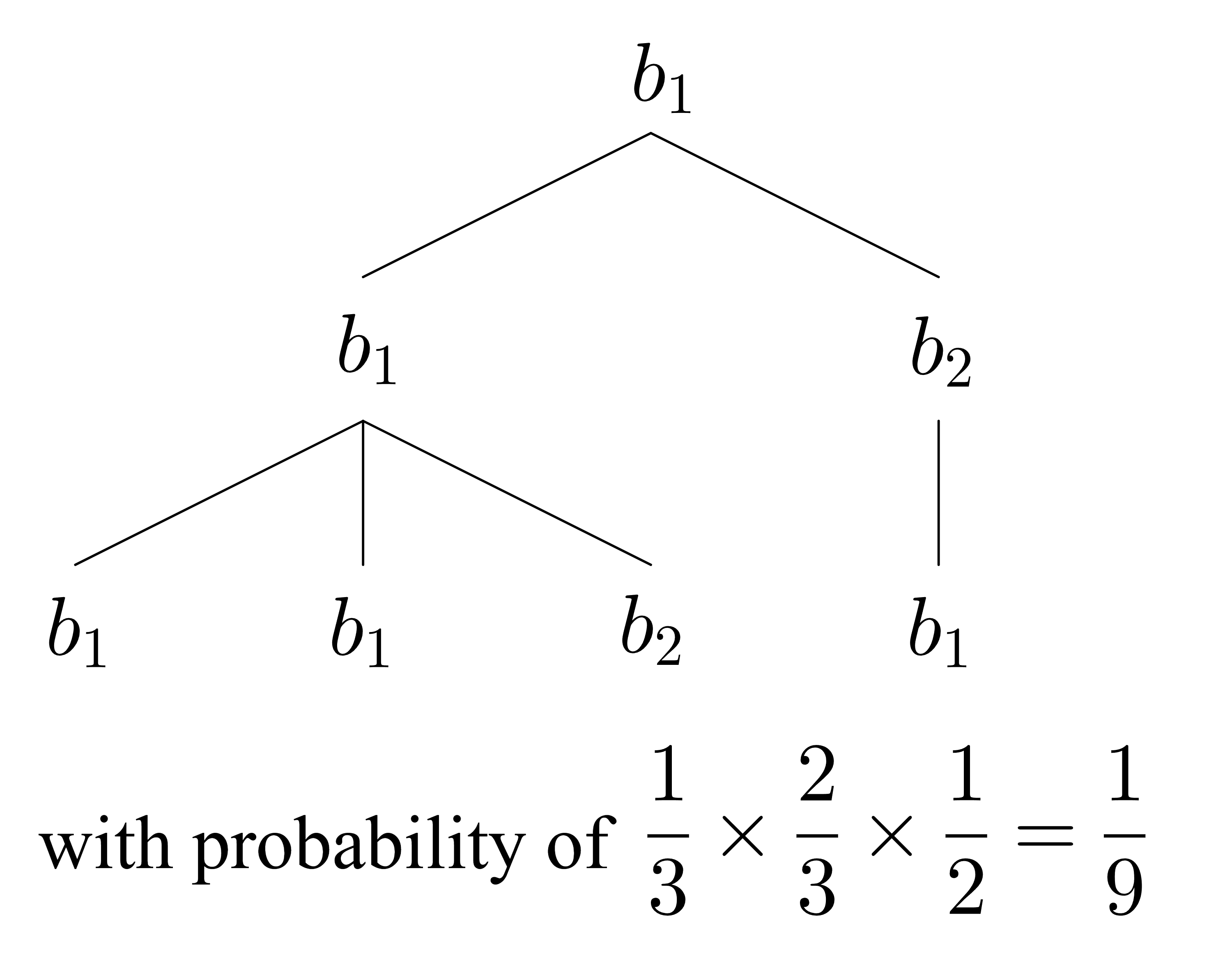} 
	\caption{A potential $2$-pattern} 
 \label{Fig: potential 2-patterns}
\end{figure}

\begin{figure}[H] 
	\centering 
	\includegraphics[width=0.4\textwidth]{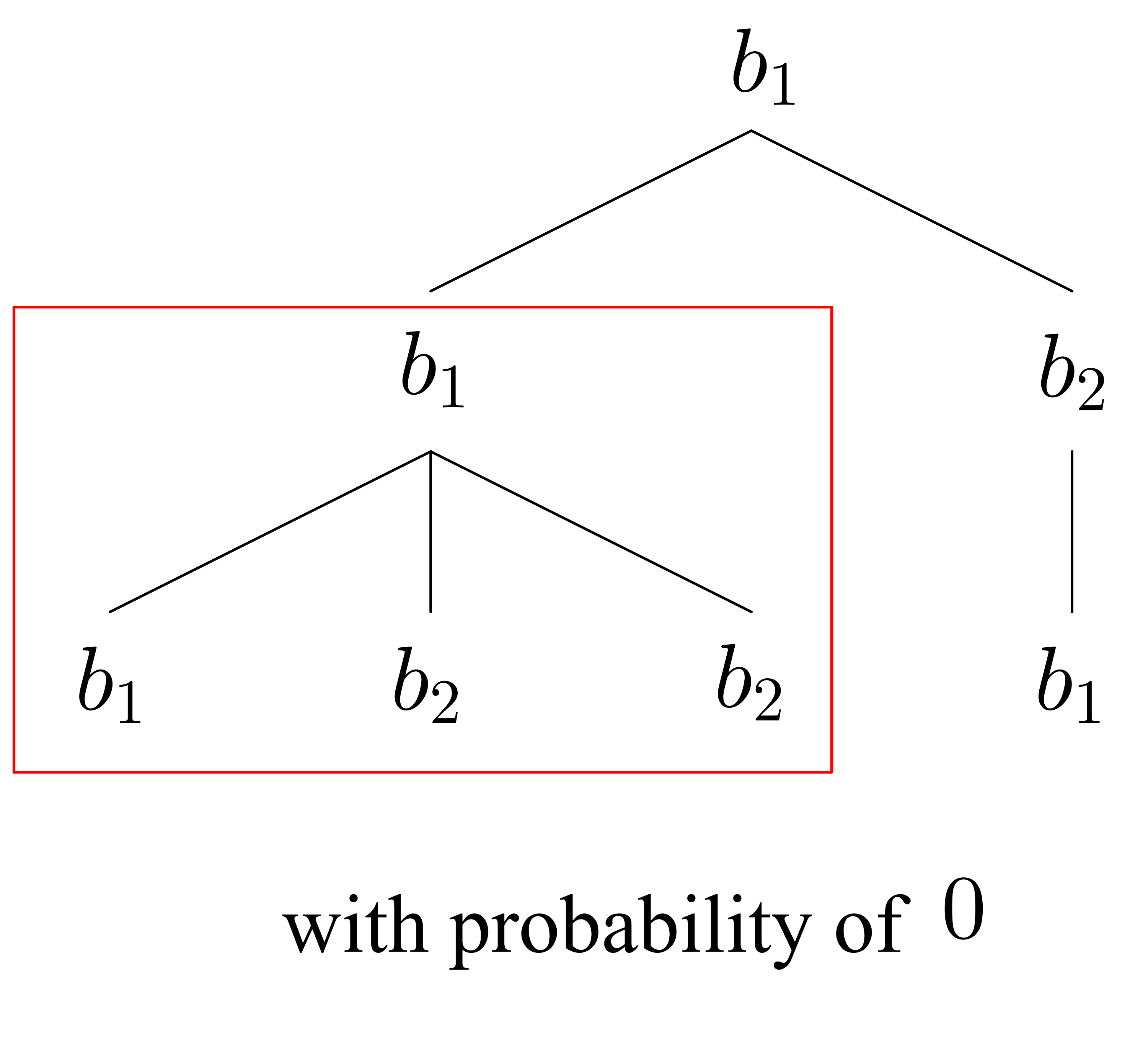} 
	\caption{Not a potential $2$-pattern } 
 \label{Fig: non-potential 2-patterns}
\end{figure}
\end{example}

Let $\mathbf{B}^{[k]}_{b_i}$ be the collection of all potential $k$-patterns initiated with a root of type $b_i$, $i=1,2,\cdots, K$. The assumption that the supports of the $k$-patterns are subsets of the conventional $d$-tree $\mathcal{T}_d$ (see the definition in Section \ref{sec: top}) ensures that each $\mathbf{B}^{[k]}_{b_i}$ only contains a finite number, say $n_i$, of potential $k$-patterns, and then there are exactly $K^{\#}=n_1+n_2+\cdots+n_K$ elements in the union
$$
\begin{array}{c}
\mathbf{B}^{[k]}\equiv\displaystyle{\bigcup_{i=1}^K \mathbf{B}^{[k]}_{b_i}}=\{\mathbf{b}_1, \mathbf{b}_2, \cdots, \mathbf{b}_{K^{\#}}\}.
\end{array}
$$

 By following Ban and et al. \cite{ban2023mathematical}, one can construct a random $k$-spread model $\mathcal{R}^{[k]}$ from the given random $1$-spread model $\{\mathbf{Z}_n\}_{n\geq 0}$ with type set $\mathcal{B}$ in a natural way following the spread distribution $\{p^{(b_i)}(\cdot)\}_{i=1}^{K}$. We refer the readers to \cite{ban2023mathematical} for more details about the $k$-spread model $\mathcal{R}^{[k]}$ and illustrations are provided in Figure \ref{Fig: relabeled} in Example \ref{ex: r2} as well as Figure \ref{Fig: 2b_induced pattern} in Section 4.2.3 in the current work. Furthermore, this natural random $k$-spread model $\mathcal{R}^{[k]}$ will induce another random $1$-spread model
$$
\begin{array}{c}
\mathbf{Z}^{[k]}_n=(Z^{[k]}_{n,1}, Z^{[k ]}_{n,2}, \cdots,Z^{[k]}_{n,K^{\#}})
\end{array}
$$
with type set $\mathbf{B}^{[k]}$. This random $1$-spread model $\{\mathbf{Z}^{[k]}_n\}_{n\geq 0}$ is an $K^{\#}$-type branching process and each of its types is a potential $k$-pattern of the original spread model $\{\mathbf{Z}_n\}_{n\geq 0}$ and we call $\{\mathbf{Z}^{[k]}_n\}_{n\geq 0}$ the random $1$-spread model induced from $\{\mathbf{Z}_n\}_{n\geq 0}$ via the natural random $k$-spread model $\mathcal{R}^{[k]}$. The corresponding initial distribution and spread distribution $\{p^{(\mathbf{b}_i)}(\cdot)\}_{i=1}^{K^{\#}}$ of  $\{\mathbf{Z}_n^{[k]}\}_{n\geq 0}$ are determined by the probabilities of occurring for the potential $k$-patterns according to the original spread distribution $\{p^{(b_i)}(\cdot)\}_{i=1}^{K}$ of $\{\mathbf{Z}_n\}_{n\geq 0}$. In addition, it is proved as a lemma in Ban and et al. \cite{ban2023mathematical} that
$$
\begin{array}{c}
|\mathbf{Z}^{[k]}_n|=Z^{[k]}_1+Z^{[k]}_2+\cdots+Z^{[k]}_{K^\#}\rightarrow \infty \hspace{1cm} \textrm{a.s.}
\end{array}
$$
That is, the population of this induced random spread model does not go extinct with probability $1$.

Moreover, it is clear that any $k$-block code $\Phi^{[k]}:\mathcal{P}_k\rightarrow \mathcal{A}$ respect to the original random $1$-spread model $\{\mathbf{Z}_n\}_{n\geq 0}$ with type set $\mathcal{B}$ can be considered as a $0$-block code $\mathbf{\Phi}^{[k]}:\mathbf{B}^{[k]}\rightarrow \mathcal{A}$ respect to the $1$-spread model $\{\mathbf{Z}^{[k]}_n\}_{n\geq 0}$ with type set $\mathbf{B}^{[k]}$. Here, we call $\mathbf{\Phi}^{[k]}$ \emph{ the associated $0$-block code of the $1$-block code $\Phi^{[k]}$}. Therefore, according to Section \ref{r_1s0b model}, this random $1$-spread model $\{\mathbf{Z}^{[k]}_n\}_{n\geq 0}$  together with the associated $0$-block code $\mathbf{\Phi}^{[k]}:\mathbf{B}^{[k]}\rightarrow \mathcal{A}$ induce a projected spread model denoted by
$$
\begin{array}{c}
\mathbf{Z}^{[k],\Phi}_n=(Z^{[k],\Phi}_{n,1}, Z^{[k],\Phi}_{n,2}, \cdots,Z^{[k],\Phi}_{n,K'})
\end{array}
$$
where $Z^{[k],\Phi}_{n,j}$ is the number of individuals of explicit type $a_j$ in the $n$th generation of $\{\mathbf{Z}_n^{[k],\Phi}\}_{n\geq 0}$, $j=1,2,\cdots, K'$ and we call $\{\mathbf{Z}^{[k],\Phi}_n\}_{n\geq 0}$ \emph{ the random projected spread model induced from the random $1$-spread model $\{\mathbf{Z}_n\}_{n\geq 0}$ and the $k$-block code $\Phi^{[k]}$ }.

Now, we will explain in more details with the following example how to induce a random $1$-spread model  $\{\mathbf{Z}_n^{[k]}\}_{n\geq 0}$ and then a random projected spread model  $\{\mathbf{Z}_n^{[k],\Phi}\}_{n\geq 0}$ from an original random $1$-spread model  $\{\mathbf{Z}_n\}_{n\geq 0}$ when $k=1$, and a similar procedure can be adopted for other cases when $k> 1$. 

\begin{example}\label{ex: r2}
We consider the same random $1$-spread model $\{\mathbf{Z}_n\}_{n\geq 0}$ in Example \ref{ex: r1} with type set $\mathcal{B}=\{b_1,b_2\}$ and spread distribution $\{p^{(b_i)}(\cdot)\}_{i=1}^{2}$, where
$$
\begin{array}{c}
p^{(b_1)}(1,1)=\frac{1}{3},\hspace{0.3cm} p^{(b_1)}(2,1)=\frac{2}{3},\hspace{0.3cm} p^{(b_2)}(1,0)=\frac{1}{2},\hspace{0.3cm} p^{(b_2)}(1,1)=\frac{1}{2}.
\end{array}
$$
In this case, all the potential $1$-patterns are listed and labeled as $\mathbf{b}_1, \mathbf{b}_2, \mathbf{b}_3$, and $\mathbf{b}_4$ in Figure \ref{Fig: label}. So, we have $\mathbf{B}^{[1]}=\{\mathbf{b}_1, \mathbf{b}_2, \mathbf{b}_3,\mathbf{b}_4\}$. Then, the induced random $1$-spread model $\{\mathbf{Z}^{[1]}_n\}_{n\geq 0}$ with type set $\mathbf{B}^{[1]}$ has the spread distribution  $\{p^{(\mathbf{b}_i)}(\cdot)\}_{i=1}^{4}$ where 
 $$
 \begin{array}{c}
p^{(\mathbf{b}_1)}(1,0,1,0)=\frac{1}{3}\times\frac{1}{2}=\frac{1}{6}
 \end{array}
 $$
 and
 $$
\begin{array}{lll}
 p^{(\mathbf{b}_1)}(1,0,0,1)=\frac{1}{6},&p^{(\mathbf{b}_1)}(0,1,1,0)=\frac{1}{3},&p^{(\mathbf{b}_1)}(0,1,0,1)=\frac{1}{3};\vspace{0.2cm}\\
p^{(\mathbf{b}_2)}(2,0,1,0)=\frac{1}{18},& p^{(\mathbf{b}_2)}(2,0,0,1)=\frac{1}{18},&p^{(\mathbf{b}_2)}(1,1,1,0)=\frac{2}{9},\vspace{0.2cm}\\
p^{(\mathbf{b}_2)}(1,1,0,1)=\frac{2}{9},&p^{(\mathbf{b}_2)}(0,2,1,0)=\frac{2}{9},&p^{(\mathbf{b}_2)}(0,2,0,1)=\frac{2}{9};\vspace{0.2cm}\\
p^{(\mathbf{b}_3)}(1,0,0,0)=\frac{1}{3},&p^{(\mathbf{b}_3)}(0,1,0,0)=\frac{2}{3};\vspace{0.2cm}&\\
p^{(\mathbf{b}_4)}(1,0,1,0)=\frac{1}{6},&p^{(\mathbf{b}_4)}(1,0,0,1)=\frac{1}{6},&p^{(\mathbf{b}_4)}(0,1,1,0)=\frac{1}{3},\vspace{0.2cm}\\
p^{(\mathbf{b}_4)}(0,1,1,0)=\frac{1}{3}.&&
\end{array}
$$
Moreover, the initial distribution of $\{\mathbf{Z}^{[1]}_n\}$ given $\mathbf{Z}_0=\mathbf{1}_{b_1}$ is the following:
$$
\begin{array}{c}
P(\mathbf{Z}^{[1]}=\mathbf{b}_1)=\frac{1}{3}, P(\mathbf{Z}^{[1]}=\mathbf{b}_2)=\frac{2}{3},P(\mathbf{Z}^{[1]}=\mathbf{b}_3)=P(\mathbf{Z}^{[1]}=\mathbf{b}_4)=0
\end{array}
$$
and the initial distribution of $\{\mathbf{Z}^{[1]}_n\}$ given $\mathbf{Z}_0=\mathbf{1}_{b_2}$ becomes:
$$
\begin{array}{c}
P(\mathbf{Z}^{[1]}=\mathbf{b}_1)=P(\mathbf{Z}^{[1]}=\mathbf{b}_2)=0, P(\mathbf{Z}^{[1]}=\mathbf{b}_3)=\frac{1}{2}, P(\mathbf{Z}^{[1]}=\mathbf{b}_4)=\frac{1}{2}.
\end{array}
$$

Every possible realization of the original random $1$-spread model $\{\mathbf{Z}_n\}_{n\geq 0}$ is associated with a possible realization of the induced random $1$-spread model $\{\mathbf{Z}^{[1]}_n\}_{n\geq 0}$. Figure \ref{Fig: partial realization} is the first four levels of a possible realization of $\{\mathbf{Z}_n\}_{n\geq 0}$ and, in Figure \ref{Fig: relabeled}, it gives an illustration of how the realization in Figure \ref{Fig: partial realization} is associated with a realization (the red tree) of $\{\mathbf{Z}^{[1]}_n\}_{n\geq 0}$.

\begin{figure}[H] 
	\centering 
	\includegraphics[width=0.8\textwidth]{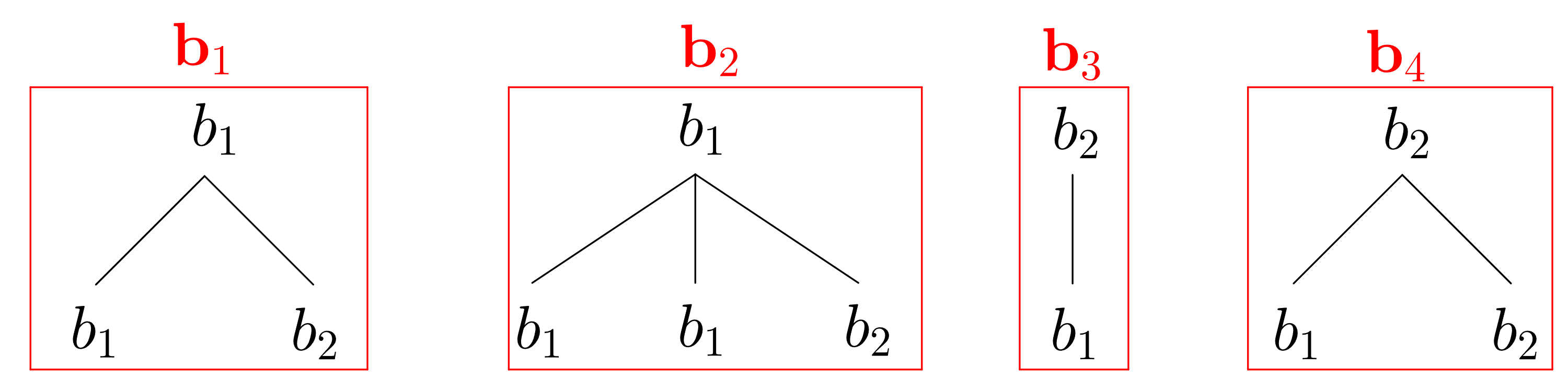} 
	\caption{$\mathbf{B}^{[1]}=\{\mathbf{b}_1, \mathbf{b}_2, \mathbf{b}_3,\mathbf{b}_4\}$ } 
  \label{Fig: label}
\end{figure}

\begin{figure}[H] 
	\centering 
	\includegraphics[width=0.5\textwidth]{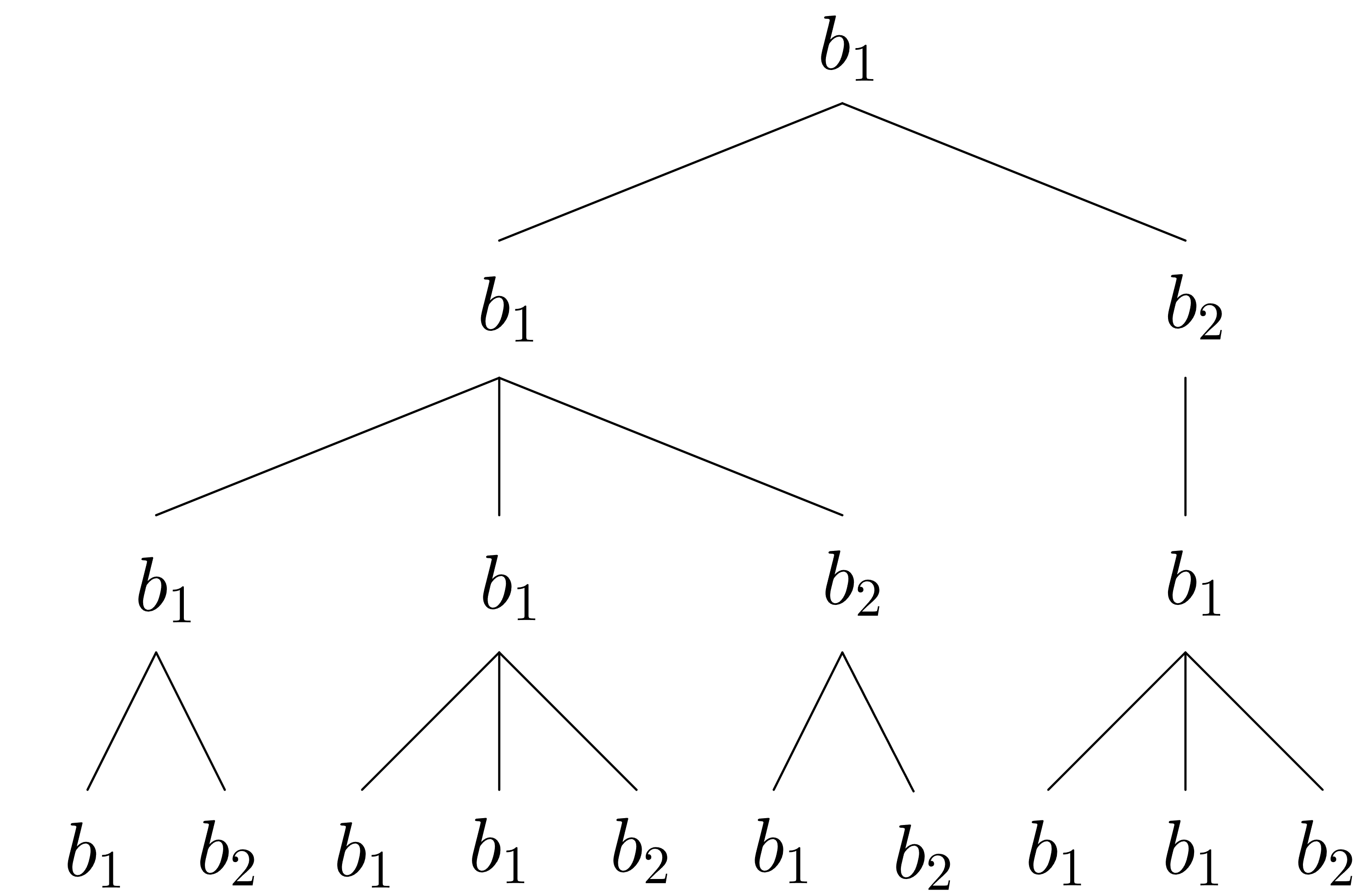} 
	\caption{The first four levels of a possible realization of $\{\mathbf{Z}_n\}_{n\geq 0}$} 
 \label{Fig: partial realization}
\end{figure}

\begin{figure}[H] 
	\centering 
	\includegraphics[width=0.8\textwidth]{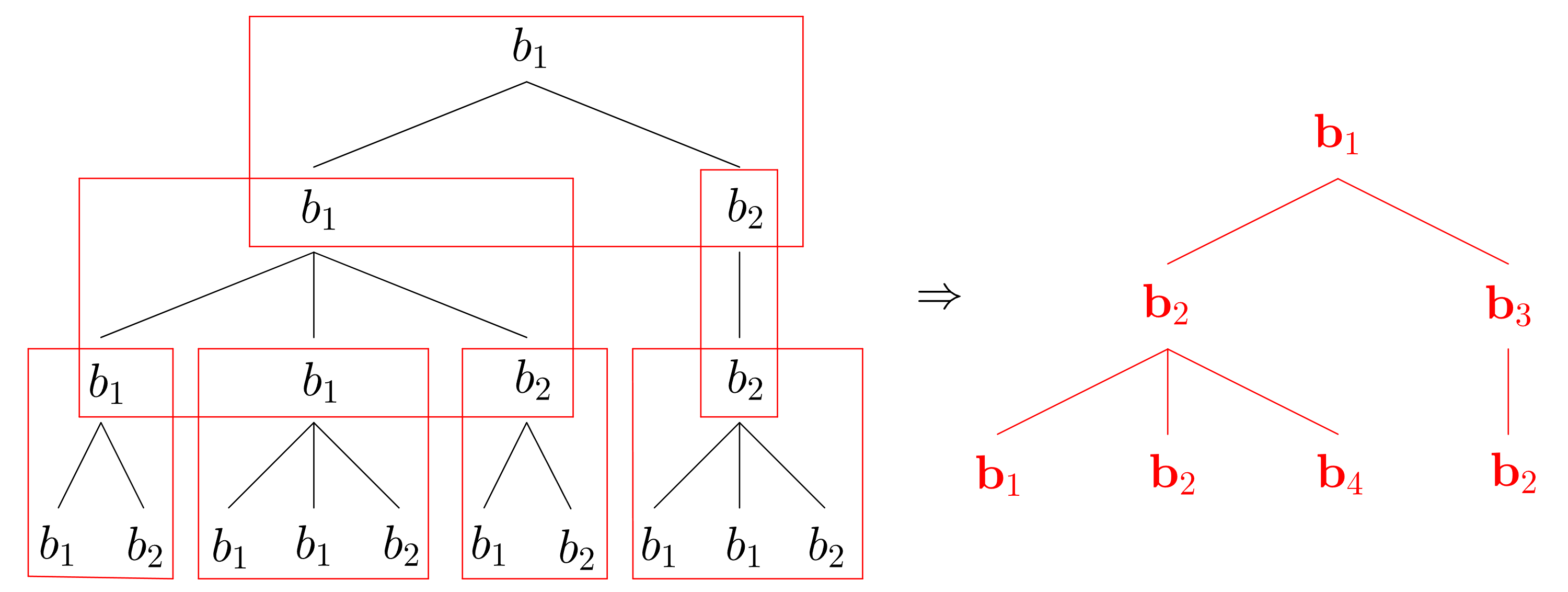} 
	\caption{An illustration of how to induce $\{\mathbf{Z}^{[1]}_n\}_{n\geq 0}$ from $\{\mathbf{Z}_n\}_{n\geq 0}$ via $\mathcal{R}^{[k]}$ } 
\label{Fig: relabeled}
\end{figure}

Now, we take another type set $\mathcal{A}=\{a, b\}$ type set and a $1$-block code $\Phi^{[1]}: \mathcal{P}_1\rightarrow \mathcal{A} $. This $1$-block code can also be viewed as a $0$-block code $\mathbf{\Phi}^{[1]}: \mathbf{B}^{[1]}\rightarrow \mathcal{A} $. Assume that 
$$
\begin{array}{c}
\mathbf{\Phi}^{[1]}(\mathbf{b_1})=\mathbf{\Phi}^{[1]}(\mathbf{b_3})=a \hspace{0.5cm} \textrm{ and } \hspace{0.5cm} \mathbf{\Phi}^{[1]}(\mathbf{b_2})=\mathbf{\Phi}^{[1]}(\mathbf{b_4})=b.
\end{array}
$$
 So, by mapping each type in $\mathbf{B}^{[1]}$ to its corresponding image in $\mathcal{A}$, $\mathbf{\Phi}^{[1]}$ transformed realizations of $\{\mathbf{Z}^{[1]}_n\}_{n\geq 0}$ into realizations of $\{\mathbf{Z}_n^{[1],\Phi}\}_{n\geq 0}$. Figure \ref{Fig: new} illustrates this transformation by $\mathbf{\Phi}^{[1]}$ for the realization (the red tree) of $\{\mathbf{Z}^{[1]}_n\}_{n\geq 0}$ in Figure \ref{Fig: relabeled} and the green tree is a realization of the random projected spread model $\{\mathbf{Z}_n^{[1],\Phi}\}_{n\geq 0}$. As we can see from Figure \ref{Fig: new}, the realizations of $\{\mathbf{Z}_n^{[1],\Phi}\}_{n\geq 0}$ has the same tree-structure as their corresponding pre-images but the nodes are labeled with types from the explicit type set $\mathcal{A}$ instead of the hidden type set $\mathcal{B}$.

\begin{figure}[H] 
	\centering 
	\includegraphics[width=0.8\textwidth]{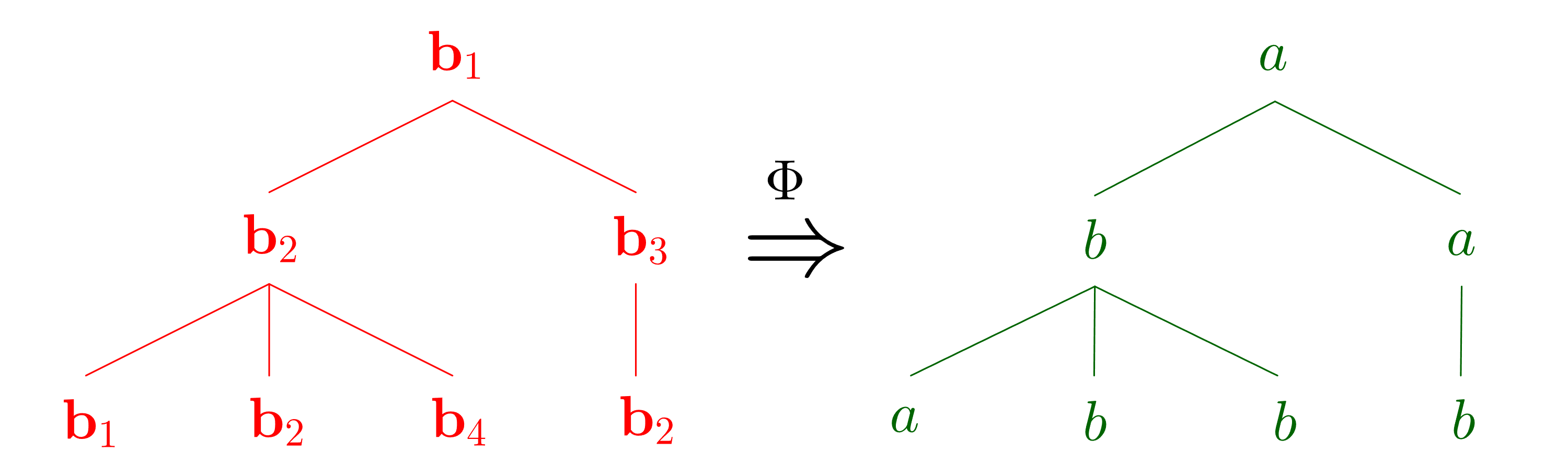} 
	\caption{An illustration of how to induce $\{\mathbf{Z}^{[1],\Phi}_n\}_{n\geq 0}$ from $\{\mathbf{Z}^{[1]}_n\}_{n\geq 0}$ } 
\label{Fig: new}
\end{figure}
\end{example}

We now can define the spread rate of an explicit type $a_j\in\mathcal{A}$ in the projected spread model $\{\mathbf{Z}^{[k],\Phi}_n\}_{n \geq 0}$ . If the original random $1$-spread model $\{\mathbf{Z}_n\}_{n\geq 0}$ is initiated with an individual of type $b_i\in \mathcal{B}$ and projected by the $k$-block code $\Phi^{[k]}:\mathcal{P}_k\rightarrow \mathcal{A}$, then the spread rate of the explicit type $a_j$ in $\{\mathbf{Z}_n^{[k],\Phi}\}_{n\geq 0}$ is defined as
$$
\begin{array}{c}
s_{b_{i}}(a_j;\{\mathbf{Z}_n\}_{n\geq 0},\Phi^{[k]})=\displaystyle{\lim_{n\rightarrow \infty}\frac{Z^{[k],\Phi}_{n,j}}{\sum\limits_{j=1}^{K'} Z^{[k],\Phi}_{n,j}}}
\end{array}
 $$
where $i=1,2,\cdots,K$ and  $j=1,2,\cdots,K'$. We also define
$$
s_{b_i}(a_j;\{\mathbf{Z}_n\}_{n\geq 0},\Phi^{[k]},\{k_{n}\}_{n=1}^{\infty })=\displaystyle{\lim_{n\rightarrow \infty }
\frac{\sum\limits_{r=s_n+1}^{s_{n+1}}Z^{[k],\Phi}_{r,j}}{\sum\limits_{r=s_n+1}^{s_{n+1}}\sum\limits_{j=1}^{K'}Z^{[k],\Phi}_{r,j}}}.
$$
where $
\{k_{n}\}_{n=1}^{\infty }\subseteq \mathbb{N}$ and $ s_{n}=\sum_{i=1}^{n}k_{i}$.

 Note that, when $k=0$, i.e., the original $1$-spread model is projected by a $0$-block code,  the spread mean matrices $M$ of  $\{\mathbf{Z}_n\}_{n\geq 0}$ provides us information to determine the growth rate of the population and find the spread rates of explicit types, as shown in Theorem \ref{thm_random 1s0b}. When $k\geq 1$, the spread mean matrix $\mathbf{M}^{[k+1]}$ of the induced random $1$-spread model $\{\mathbf{Z}^{[k]}_n\}_{n\geq 0}$ plays the same role in allowing us to find the growth rate and the spread rate of explicit types. However, $\mathbf{M}^{[k+1]}$ is an $K^{\#}\times K^{\#}$ matrix which may not be of the same size as the spread mean matrix $M$ of the original random $1$-spread model $\{\mathbf{Z}_n\}_{n\geq 0}$ and this is a difference between the topological cases (refer to Lemma \ref{Lemma:xi-matrices }) and the random cases.


\begin{lemma}\label{lem: jlogj}
Let $\{\mathbf{Z}^{[k]}_n\}_{n\geq 0}$ be the $1$-random spread model induced from $\{\mathbf{Z}_n\}_{n\geq 0}$ via the natural random $k$-spread model $\mathcal{R}^{[k]}$. Then, for all $i,j=1,2,\cdots, K^\#$,
$$
\begin{array}{c}
E(Z^{[k]}_{1,j}\log Z^{[k]}_{1,j}|\mathbf{Z}^{[k]}_0=\mathbf{1}_{\mathbf{b}_i})<\infty.
\end{array}
$$
\end{lemma}
\begin{proof} The result directly follows from the fact that the support of every potential $k$-pattern, $k\geq 1$, is a subset of the conventional $d$-tree $\mathbf{T}_d$.
\end{proof}

\begin{theorem}\label{thm_random 1skb}
Let $\{\mathbf{Z}_n\}_{n\geq 0}$ be a random $1$-spread model with the type set $\mathcal{B}=\{b_i\}_{i=1}^K$. Let $\mathcal{A}=\{a_i\}_{i=1}^{K'}$ be another type set and $\Phi^{[k]}:\mathcal{P}_k\rightarrow\mathcal{A}$ be a $k$-block code where $k$ is a positive integer. Let $\{\mathbf{Z}^{[k],\Phi}_n\}_{n\geq 0}$ be the random projected spread model induced from $\{\mathbf{Z}_n\}_{n\geq 0}$ and $\Phi^{[k]}$. Let $\mathbf{M}^{[k+1]}$ be the spread mean matrix of the random $1$-spread model $\{\mathbf{Z}^{[k]}_n\}_{n\geq 0}$ induced from $\{\mathbf{Z}_n\}_{n\geq 0}$ via $\mathcal{R}^{[k]}$.  Suppose that $\mathbf{\rho}^{[k]}>1$ is the maximal eigenvalue of $\mathbf{M}^{[k+1]}$ with positive normalized left eigenvector $\mathbf{w}=(\mathbf{w}_1, \cdots, \mathbf{w}_{K^\#})$ and suppose that $
\{k_{n}\}_{n=1}^{\infty }\subseteq \mathbb{N}$ is a sequence such that $k_n=k$ for all $n$ or $k_n\rightarrow \infty$ as $n\rightarrow \infty$. Then
\begin{enumerate}
  \item [(i)] there exists a random variable $W^{[k]}$ and a vector $u$ such that
  $$
  \begin{array}{c}
  \displaystyle{\lim_{n\rightarrow \infty}\frac{\mathbf{Z}^{[k],\Phi}_n}{(\mathbf{\rho}^{[k]})^n}}= u W^{[k]} \hspace{1cm} \textrm{ a.s. }
  \end{array}
  $$
  \item [(ii)] for any $b_{i_0}\in \mathcal{B}$, on the event of non-extinction given that $\mathbf{Z}_0=\mathbf{1}_{b_{i_0}}$, the spread rate is
$$
\begin{array}{c}
s_{b_{i_0}}(a_j;\{\mathbf{Z}_n\}_{n\geq 0},\Phi^{[k]},\{k_{n}\}_{n=1}^{\infty })=\displaystyle{\sum_{i: \Phi^{[k]}(\mathbf{b}_i)=a_j}\mathbf{w}_i}\hspace{1cm} \textrm{ a.s. }
\end{array}
$$
and is independent of $b_{i_0}$.
\end{enumerate}
\end{theorem}

\begin{proof}
\begin{enumerate}
\item [(i)] Since  $\{\mathbf{Z}^{[k],\Phi}_n\}_{n\geq 0}$ can also be considered as the projected spread model induced from the random $1$-spread model  $\{\mathbf{Z}^{[k]}_n\}_{n\geq 0}$ and the associate $0$-block code $\mathbf{\Phi}^{[k]}$, together with Lemma \ref{lem: jlogj}, by the similar lines for proving Theorem \ref{thm_random 1s0b} (i), we can show that there exists a random variable $W^{[k]}$ and a vector $u$ such that
  $$
  \begin{array}{c}
  \displaystyle{\lim_{n\rightarrow \infty}\frac{\mathbf{Z}^{[k],\Phi}_n}{(\mathbf{\rho}^{[k]})^n}}= u W^{[k]} \hspace{1cm} \textrm{ a.s. }
  \end{array}
  $$
where $u=(u_1,u_2,\cdots,u_{K'})$ and
$$
\begin{array}{c}
u_j=\sum\limits_{i: \mathbf{\Phi^{[k]}}(\mathbf{b}_i)=a_j} \mathbf{w}_i
\end{array}
$$
for $j=1,2,\cdots, K'$.

\item [(ii)] Let $E_{i_0}$ be the event of non-extinction for $\{\mathbf{Z}_n\}_{n\geq 0}$ given that $\mathbf{Z}_0=\mathbf{1}_{b_{i_0}}$, $b_{i_0}\in \mathcal{B}$, and let $\mathbf{E}_r$ be the event of non-extinction for $\{\mathbf{Z}^{[k]}_n\}_{n\geq 0}$ given that $\mathbf{Z}^{[k]}_0=\mathbf{1}_{\mathbf{b}_r}$, $\mathbf{b}_r\in\mathbf{B}^{[k]}$. For every $\omega\in E_{i_0}$, we have that
$$
\begin{array}{c}
|\mathbf{Z}_n(\omega)|=\displaystyle{\sum\limits_{j=1}^K Z_{n,j}(\omega)}\rightarrow \infty \hspace{0.3cm} \textrm{ as } n\rightarrow \infty
\end{array}
$$
and, by the construction of $\{\mathbf{Z}^{[k]}_n\}_{n\geq 0}$ from $\{\mathbf{Z}_n\}_{n\geq 0}$, it implies that
$$
\begin{array}{c}
|\mathbf{Z}^{[k]}_n(\omega)|=\displaystyle{\sum\limits_{j=1}^K Z^{[k]}_{n,j}(\omega)}\rightarrow \infty \hspace{0.3cm} \textrm{ as } n\rightarrow \infty.
\end{array}
$$
Therefore, $E_{i_0} \subset \displaystyle{\cup_{\mathbf{b}_r\in \mathbf{B}^{[k]}_{i_0}}\mathbf{E}_r}$.

In addition, it follows from Theorem \ref{thm_random 1s0b} (ii) that, for any $\mathbf{b}_{r}\in \mathbf{B}^{[k]}$, on the event of non-extinction $\mathbf{E}_{r}$ for $\{\mathbf{Z}^{[k]}_n\}_{n\geq 0}$, given that $\mathbf{Z}^{[k]}_0=\mathbf{1}_{\mathbf{b}_r}$,
$$
\begin{array}{c}
s_{\mathbf{b}_r}(a_j;\{\mathbf{Z}^{[k]}_n\}_{n\geq 0},\mathbf{\Phi}^{[k]})=\displaystyle{\lim_{n\rightarrow \infty}\frac{Z^{[k],\Phi}_{n,j}}{\sum\limits_{j=1}^{K'} Z^{[k],\Phi}_{n,j}}=\displaystyle{\sum_{i: \mathbf{\Phi}^{[k]}(\mathbf{b}_i)=a_j}\mathbf{w}_i}}\hspace{0.5cm} \textrm{ a.s. }
\end{array}
$$
which is independent of $\mathbf{b}_r$, where $I_D$ is the indicator function of the event $D$. So, since the event
$$
\begin{array}{c}
\big\{\omega:\mathbf{Z}_0(\omega)=\mathbf{1}_{b_{i_0}}(\omega)\big\}=\bigcup\limits_{\mathbf{b}_r\in\mathbf{B}^{[k]}_{i_0}}\big\{\omega:\mathbf{Z}^{[k]}_0(\omega)=\mathbf{1}_{\mathbf{b}_r}(\omega)\big\}
\end{array}
$$ 
is a disjoint union and $\mathbf{B}^{[k]}_{i_0}$ is a finite set, by Lemma \ref{Lma: 2}, we have that
$$
\begin{array}{rl}
&s_{b_{i_0}}(a_j;\{\mathbf{Z}_n\}_{n\geq 0},\Phi^{[k]})\\
=&\displaystyle{\lim_{n\rightarrow \infty}\frac{Z^{[k],\Phi}_{n,j}}{\sum\limits_{j=1}^{K'} Z^{[k],\Phi}_{n,j}}}=\displaystyle{\lim_{n\rightarrow \infty}\frac{\sum\limits_{\mathbf{b}_r\in \mathbf{B}^{[k]}_{i_0}}Z^{[k],\Phi}_{n,j}I_{\{\mathbf{Z}^{[k]}_0=\mathbf{1}_{\mathbf{b}_r}\}}}{\sum\limits_{j=1}^{K'} \sum\limits_{\mathbf{b}_r\in \mathbf{B}^{[k]}_{i_0}}Z^{[k],\Phi}_{n,j}I_{\{\mathbf{Z}^{[k]}_0=\mathbf{1}_{\mathbf{b}_r}\}}}}\\
=&\displaystyle{\lim_{n\rightarrow \infty}\frac{\sum\limits_{\mathbf{b}_r\in \mathbf{B}^{[k]}_{i_0}}Z^{[k],\Phi}_{n,j}I_{\{\mathbf{Z}^{[k]}_0=\mathbf{1}_{\mathbf{b}_r}\}}}{ \sum\limits_{\mathbf{b}_r\in \mathbf{B}^{[k]}_{i_0}}\sum\limits_{j=1}^{K'}Z^{[k],\Phi}_{n,j}I_{\{\mathbf{Z}^{[k]}_0=\mathbf{1}_{\mathbf{b}_r}\}}}=\sum\limits_{i: \mathbf{\Phi}^{[k]}(\mathbf{b}_i)=a_j}\mathbf{w}_i} \hspace{0.3cm} \textrm{a.s. on $E_{i_0}$.}
\end{array}
 $$
 Moreover, it follows from Lemma \ref{Lma: 2} again that
$$
\begin{array}{rl}
&s_{b_{i_0}}(a_j;\{\mathbf{Z}_n\}_{n\geq 0},\Phi^{[k]},\{k_{n}\}_{n=1}^{\infty })\\
=&\displaystyle{\lim_{n\rightarrow \infty }
\frac{\sum\limits_{r=s_n+1}^{s_{n+1}}Z^{[k],\Phi}_{r,j}}{\sum\limits_{r=s_n+1}^{s_{n+1}}\sum\limits_{j=1}^{K'}Z^{[k],\Phi}_{r,j}}}=\displaystyle{\lim_{n\rightarrow \infty}\frac{Z^{[k],\Phi}_{n,j}}{\sum\limits_{j=1}^{K'} Z^{[k],\Phi}_{n,j}}=\sum\limits_{i: \Phi(\mathbf{b}_i)=a_j} \mathbf{w}_i}  \hspace{0.3cm} \textrm{a.s. on $E_{i_0}$.}
\end{array}
$$
\end{enumerate}
\end{proof}

\section{Examples with numerical results}\label{sec: num}

In this section, we constructed some examples to help illustrate the results of topological models described in Section 2 and the random models described in Section 3 of this paper.

\subsection{Examples for Topological models}
In this subsection, we construct three examples corresponding to the cases where $m-1=k$, $m-1<k$, and $m-1>k$, respectively, to illustrate Theorem \ref{Thm: 2}, Theorem \ref{Thm: 4}, and Theorem \ref{Thm: 3} in Section 2.
\subsubsection{The case where $m-1=k$}
We first consider $(m,k)=(1,0).$ Let $\mathcal{B}=\{a^1,a^2,b\}$ and $\mathcal{A}=\{a,b\}.$ Assume the 1-spread model $\mathcal{S}=\{p_1=(a^1;a^1,a^2,b),p_2=(a^2;a^1,a^2,b),p_3=(b;a^1,a^2)\}.$  Define the 0-block code $\Phi:\mathcal{B}\to\mathcal{A}$ by $\Phi(a^i)=a$ and $\Phi(b)=b.$ Then the associated $\xi$-matrix is 
\begin{align*}
    M&=\begin{pmatrix}
    1&1&1\\
    1&1&1\\
    1&1&0
\end{pmatrix}\\
&=\begin{pmatrix}
    \frac{2\sqrt{3}+3}{6}&-1&\frac{-\sqrt{3}+1}{2}\\
    \frac{2\sqrt{3}+3}{6}&1&\frac{-\sqrt{3}+1}{2}\\
    \frac{\sqrt{3}+3}{6}&0&1
\end{pmatrix}\begin{pmatrix}
    \sqrt{3}+1&0&0\\
    0&0&0\\
    0&0&-\sqrt{3}+1
\end{pmatrix}\begin{pmatrix}
    \frac{\sqrt{3}-1}{2}&\frac{\sqrt{3}-1}{2}&2-\sqrt{3}\\
    \frac{-1}{2}&\frac{-1}{2}&0\\
    \frac{-\sqrt{3}}{6}&\frac{-\sqrt{3}}{6}&\frac{\sqrt{3}+3}{6}
\end{pmatrix}
\end{align*}
For $b\in\mathcal{B}$ and $a\in\mathcal{A},$ then we have 
\begin{align*}
    s_{p_3}(a,\mathcal{S},\Phi,\{1\}_{n=1}^\infty)&=\lim_{n\to\infty}\frac{\textbf{1}_b^t M^n\textbf{1}_{\Phi^{-1}(a)}}{\textbf{1}_b^t M^n\textbf{1}}\\
    &=\lim_{n\to\infty}\frac{\frac{2\sqrt{3}}{6}(\sqrt{3}+1)^n+\frac{-2\sqrt{3}}{6}(-\sqrt{3}+1)^n}{\frac{\sqrt{3}+3}{6}(\sqrt{3}+1)^n+\frac{-\sqrt{3}+3}{6}(-\sqrt{3}+1)^n}\\
    &=\sqrt{3}-1=\sum_{c\in \mathcal{B}:\Phi(c)=a}w(c),
\end{align*}
where $w=\left(\frac{\sqrt{3}-1}{2},\frac{\sqrt{3}-1}{2},2-\sqrt{3}\right)$ is the left eigenvector of the $\xi$-matrix $M$ corresponding to the maximal eigenvalue $\rho_M=\sqrt{3}+1$ of $M.$

In this case, we run a simulation (cf. Figure \ref{Fig: topological spread rate}) and find that the experimental values coincide with the theoretical values.

\begin{figure}[H] 
	\centering 
	\includegraphics[width=\textwidth]{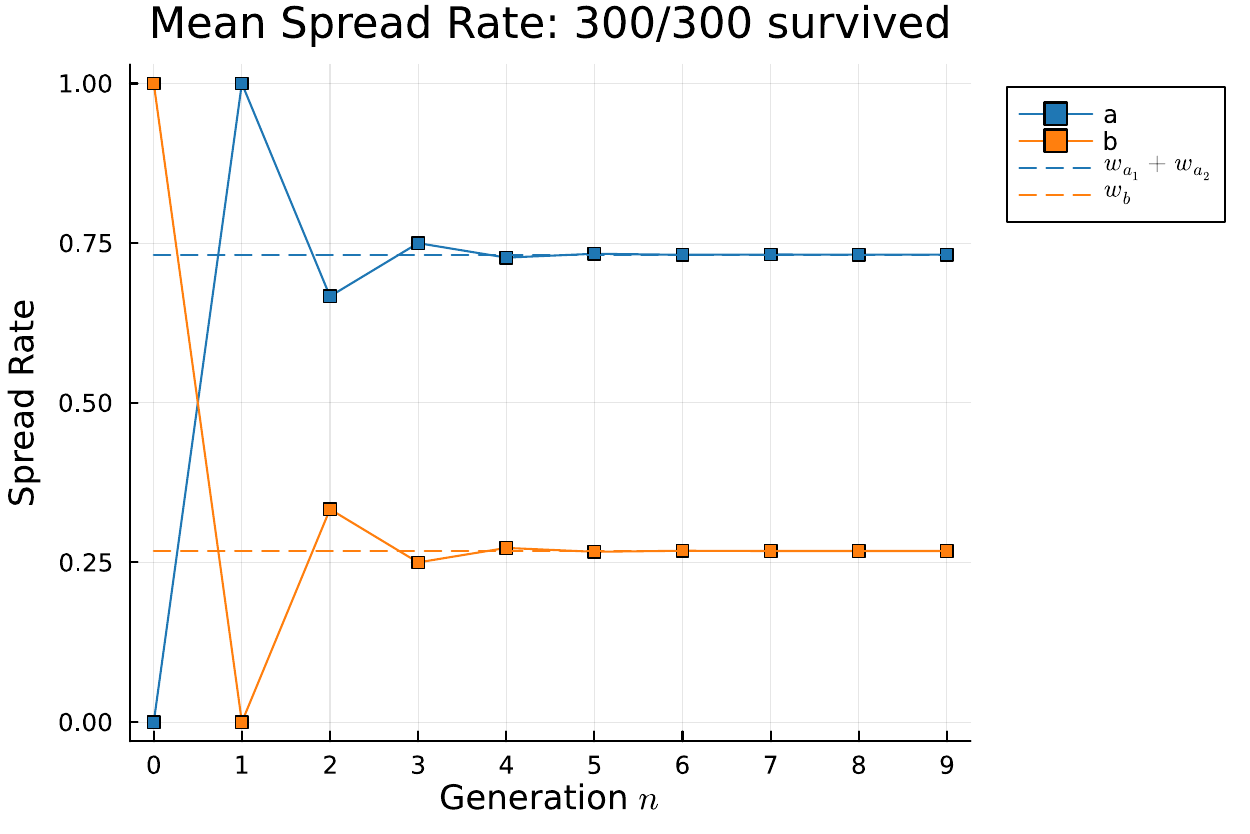} 
	\caption{Examples of a topological model with three types. The solid line represents experimental values, while the dashed line represents theoretical values. In this case, the spread rate is numerically approximated by averaging over 300 simulations.} 
  \label{Fig: topological spread rate}
\end{figure}

\subsubsection{The case where $m-1<k$}
We now consider $(m,k)=(1,1).$ Let $\mathcal{B}=\{a^1,a^2,b\}$ and $\mathcal{A}=\{a,b\}.$ Assume the 1-spread model $\mathcal{S}=\{p_1=(a^1;a^2,b),p_2=(a^2;a^1,a^2),p_3=(b;a^1,b)\}.$ Then $\mathcal{B}^{[1]}=\{\textbf{a}^1,\textbf{a}^2,\textbf{b}\}$ and 
$$
\begin{array}{c}
\mathcal{S}^{[2]}=\{\mathbf{p}_1^{[2]}=(a^1;a^2,b;a^1,a^2,a^1,b),\mathbf{p}_2^{[2]}=(a^2;a^1,a^2;a^2,b,a^1,a^2),\\\mathbf{p}_3^{[2]}=(b;a^1,b;a^2,b,a^1,b)\}.
\end{array}
$$
Define the 1-block code $\Phi^{[1]}:\mathcal{B}^{[1]}\to\mathcal{A}$ by $\Phi^{[1]}(p_1)=\Phi^{[1]}(p_2)=a$ and $\Phi^{[1]}(p_3)=b.$ Then the $\xi$-matrix of $\mathcal{S}^{[2]}$ is 
\begin{align*}
    \mathbf{M}^{[2]}&=\begin{pmatrix}
        0&1&1\\
        1&1&0\\
        1&0&1
    \end{pmatrix}\\
    &=\begin{pmatrix}
        1&0&-2\\
        1&-1&1\\
        1&1&1
    \end{pmatrix}\begin{pmatrix}
        2&0&0\\
        0&1&0\\
        0&0&-1
    \end{pmatrix}\begin{pmatrix}
        \frac{1}{3}&\frac{1}{3}&\frac{1}{3}\\
        0&\frac{-1}{2}&\frac{1}{2}\\
        \frac{-1}{3}&\frac{1}{6}&\frac{1}{6}
    \end{pmatrix}.
\end{align*}
In this case, we obtain that the $\xi$-matrix $M$ of $\mathcal{S}$ is equal to $\mathbf{M}^{[2]}$ (cf. Lemma \ref{Lma: 1}).

For $b\in \mathcal{B}$ and $a\in \mathcal{A},$ we have
\begin{align*}
    s_{p_3}(a,\mathcal{S},\Phi^{[k]},\{1\}_{n=1}^\infty)&=\sum_{\mathbf{c}:\Phi^{[1]}(\mathbf{c})=a}s_{\mathbf{p}_3}(\mathbf{c},\mathcal{S}^{[2]},\Phi^{[1]},\{1\}_{n=1}^\infty)\\
    &=\lim_{n\to\infty}\frac{\textbf{1}_\textbf{b}^t (\mathbf{M}^{[2]})^n \textbf{1}_{\textbf{a}^1}}{\textbf{1}_\textbf{b}^t (\mathbf{M}^{[2]})^n \textbf{1}}+\lim_{n\to\infty}\frac{\textbf{1}_\textbf{b}^t (\mathbf{M}^{[2]})^n \textbf{1}_{\textbf{a}^2}}{\textbf{1}_\textbf{b}^t (\mathbf{M}^{[2]})^n \textbf{1}}\\
    &=\lim_{n\to\infty}\frac{\frac{2^n}{3}+\frac{-(-1)^n}{3}}{2^n}+\lim_{n\to\infty}\frac{\frac{2^n}{3}-\frac{1}{2}+\frac{-(-1)^n}{6}}{2^n}\\
    &=\frac{1}{3}+\frac{1}{3}=\sum_{\textbf{c}\in\mathcal{B}^{[k]}:\Phi^{[1]}(\textbf{c})=a} \mathbf{w}(\textbf{c}),
\end{align*}
where $\mathbf{w}=(\frac{1}{3},\frac{1}{3},\frac{1}{3})$ is the left eigenvector of the $\xi$-matrix $\mathbf{M}^{[2]}$ corresponding to the maximal eigenvalue $\rho_{\mathbf{M}^{[2]}}=2$ of $\mathbf{M}^{[2]}.$

\subsubsection{The case where $m-1>k$}

We now consider $(m,k)=(2,0).$ Let $\mathcal{B}=\{a^1,a^2,b\}$ and $\mathcal{A}=\{a,b\}.$ Assume the $2$-spread model 
$$
\begin{array}{c}
\mathcal{S}=\{p_1=(a^1;a^1,a^2;b,a^2,a^1,a^2),p_2=(a^1;b,a^2;a^1,a^2,a^1,a^2),\\
p_3=(a^2;a^1,a^2;a^1,a^2,a^1,a^2),p_4=(b;a^1,a^2;a^1,a^2,a^1,a^2)\}.
\end{array}
$$
Then 
$$
\begin{array}{c}
\mathcal{B}^{[1]}=\{\textbf{b}_1=(a^1;a^1,a^2),\textbf{b}_2=(a^1;b,a^2),\textbf{b}_3=(a^2;a^1,a^2),\textbf{b}_4=(b;a^1,a^2)\}
\end{array}
$$
and
$$
\begin{array}{c}
\mathcal{S}^{[2]}=\{\textbf{p}_1=(\textbf{b}_1;\textbf{b}_2,\textbf{b}_3),\textbf{p}_2=(\textbf{b}_2;\textbf{b}_4,\textbf{b}_3),\textbf{p}_3=(\textbf{b}_3;\textbf{b}_1,\textbf{b}_3),\textbf{p}_4=(\textbf{b}_4;\textbf{b}_1,\textbf{b}_3)\}
\end{array}
$$
which is the collection of $1$-patterns over $\mathcal{B}^{[1]}.$ Define the 0-block code $\Phi:\mathcal{B}\to\mathcal{A}$ by $\Phi(a^i)=a$ and $\Phi(b)=b.$ Then the $\xi$-matrix of $\mathcal{S}^{[2]}$ over $\mathcal{B}^{[1]}$ is 
\[
    \mathbf{M}=\begin{pmatrix}
        0&1&1&0\\
        0&0&1&1\\
        1&0&1&0\\
        1&0&1&0
    \end{pmatrix}=P\begin{pmatrix}
        2&0&0&0\\
        0&0&0&0\\
        0&0&\frac{-1-i\sqrt{3}}{2}&0\\
        0&0&0&\frac{-1+i\sqrt{3}}{2}
    \end{pmatrix}P^{-1},\]
where
    \[P=\begin{pmatrix}
        1&1&\frac{-3-i\sqrt{3}}{2}&\frac{-3+i\sqrt{3}}{2}\\
        1&1&-1+i\sqrt{3}&-1-i\sqrt{3}\\
        1&-1&1&1\\
        1&1&1&1
    \end{pmatrix}~\text{and}~P^{-1}=\begin{pmatrix}
        \frac{2}{7}&\frac{1}{7}&\frac{1}{2}&\frac{1}{14}\\
        0&0&\frac{-1}{2}&\frac{1}{2}\\
        \frac{-3+2i\sqrt{3}}{21}&\frac{-3-5i\sqrt{3}}{42}&0&\frac{9+i\sqrt{3}}{42}\\
        \frac{-3-2i\sqrt{3}}{21}&\frac{-3+5i\sqrt{3}}{42}&0&\frac{9-i\sqrt{3}}{42}
    \end{pmatrix}.\]
Then, for $b\in\mathcal{B}$ and $a\in\mathcal{A},$ we have 
\begin{align*}
    s_{p_4}(a,\mathcal{S},\Phi,\{1\}_{n=1}^\infty)&=\sum_{c:\Phi(c)=a}s_{\textbf{p}_4}(c,\mathcal{S}^{[2]},\Phi,\{1\}_{n=0}^\infty)\\
    &=s_{\textbf{p}_4}(a^1,\mathcal{S}^{[2]},\Phi,\{1\}_{n=0}^\infty)+s_{\textbf{p}_4}(a^2,\mathcal{S}^{[2]},\Phi,\{1\}_{n=0}^\infty)\\
    &=\lim_{n\to\infty}\frac{\textbf{1}_{\textbf{b}_4}^t \mathbf{M}^n \textbf{1}_{\theta_{a^1}}}{\textbf{1}_{\textbf{b}_4}^t \mathbf{M}^n \textbf{1}}+\lim_{n\to\infty}\frac{\textbf{1}_{\textbf{b}_4}^t \mathbf{M}^n \textbf{1}_{\theta_{a^2}}}{\textbf{1}_{\textbf{b}_4}^t \mathbf{M}^n \textbf{1}}\\
    &=\lim_{n\to\infty} \frac{\frac{3}{7}2^n+\frac{-9-i\sqrt{3}}{42}\left(\frac{-1-i\sqrt{3}}{2}\right)^n+\frac{-9+i\sqrt{3}}{42}\left(\frac{-1+i\sqrt{3}}{2}\right)^n}{2^n}+\lim_{n\to\infty}\frac{\frac{1}{2}2^n}{2^n}\\
    &=\frac{3}{7}+\frac{1}{2}=\sum_{\textbf{c}:\textbf{c}\in\theta_{a^1}}w(\textbf{c})+\sum_{\textbf{c}:\textbf{c}\in\theta_{a^2}}w(\textbf{c})=\sum_{c:\Phi(c)=a}\sum_{\textbf{c}:\textbf{c}\in\theta_a}w(c),
\end{align*}
where $w=\left(\frac{2}{7},\frac{1}{7},\frac{1}{2},\frac{1}{14}\right)$ is the left eigenvector of the $\xi$-matrix $\mathbf{M}$ corresponding to the maximal eigenvalue $\rho_{\mathbf{M}}=2$ of $\mathbf{M}.$

\subsection{Examples for Random models}

In this section, we provide the numerical evidence for our theoretical results in Theorem \ref{thm_random 1s0b} and Theorem \ref{thm_random 1skb} through examples and simulations.

\subsubsection{The $0$-block code case}

Let $\{\mathbf{Z}_n\}_{n\geq 0}$ be the random 1-spread model with type set $\mathcal{B}=\{A_1,A_2, B_1, B_2, C_1\}$ and offspring distribution as follows:
$$
\begin{array}{lll}
p^{(A_1)}(0,1,0,0,1)=\frac{2}{5}, & p^{(A_1)}(1,0,3,1,0)=\frac{3}{5},&\vspace{0.1cm}\\
p^{(A_2)}(1,2,0,1,0)=\frac{1}{3}, & p^{(A_2)}(2,0,1,1,0)=\frac{1}{3},&p^{(A_2)}(0,0,2,0,1)=\frac{1}{3},\vspace{0.1cm}\\
p^{(B_1)}(0,0,3,1,1)=\frac{2}{3}, & p^{(B_1)}(1,1,1,0,1)=\frac{1}{3},&\vspace{0.1cm}\\
p^{(B_2)}(0,2,1,1,0)=\frac{1}{5}, & p^{(B_2)}(1,1,0,0,2)=\frac{4}{5},&\vspace{0.1cm}\\
p^{(C_1)}(2,0,0,0,3)=\frac{1}{2}, & p^{(C_1)}(0,1,3,1,0)=\frac{1}{2}.&
\end{array}
$$
Then, its offspring mean matrix is 
$$
M=\left(\begin{array}{ccccc}
        \frac{3}{5} & \frac{2}{5} & \frac{9}{5} & \frac{3}{5}& \frac{2}{5} \vspace{0.1cm}\\
      1 & \frac{2}{5} & 1 & \frac{2}{3}& \frac{1}{3} \vspace{0.1cm}\\
       \frac{1}{3} & \frac{1}{3} & \frac{7}{3} & \frac{2}{3}& 1\vspace{0.1cm} \\
        \frac{4}{5} & \frac{6}{5} & \frac{1}{5} & \frac{1}{5}& \frac{8}{5} \vspace{0.1cm}\\
         1 & \frac{1}{2} & \frac{3}{2} & \frac{1}{2}& \frac{3}{2}
    \end{array}
\right)
$$
and its maximal eigenvalue is $\rho\approx 4.38368$ with corresponding normalized left eigenvector 
$$
\begin{array}{ll}
w&=(w_{A_1}, w_{A_2}, w_{B_1}, w_{B_2}, w_{C_1})\vspace{0.1cm}\\
&\approx(0.151791, 0.114156, 0.372625, 0.127317, 0.234111).
\end{array}
$$
Now, consider another type set $\mathcal{A}=\{A, B, C\}$ and the $0$-block code $\Phi: \mathcal{B}\rightarrow \mathcal{A}$ such that
$$
\begin{array}{c}
\Phi(A_1)=\Phi(A_2)=A, \hspace{0.3cm}\Phi(B_1)=\Phi(B_2)=B,  \hspace{0.3cm}\textrm{ and }\hspace{0.3cm}\Phi(C_1)=\Phi(A_2)=C.
\end{array}
$$
Then, according to Theorem \ref{thm_random 1s0b}, the spread rates in the random projected spread model $\{\mathbf{Z}^s_n\}_{n\geq 0}$ induced by $\{\mathbf{Z}_n\}_{n\geq 0}$ and the $0$-block code $\Phi: \mathcal{B}\rightarrow \mathcal{A}$ are
$$
\begin{array}{l}
s_{b}(A;\{\mathbf{Z}_n\}_{n\geq 0},\Phi)= w_{A_1}+w_{A_2}\approx 0.265947\vspace{0.1cm}\\
s_{b}(B;\{\mathbf{Z}_n\}_{n\geq 0},\Phi)= w_{B_1}+w_{B_2}\approx 0.499942\vspace{0.1cm}\\
s_{b}(C;\{\mathbf{Z}_n\}_{n\geq 0},\Phi)= w_{C_1}\approx 0.234111
\end{array}
$$
with probability $1$ for any $b\in \mathcal{B}=\{A_1, A_2, B_1, B_2, C_1\}$. We run a simulation for this example and the mean ratio of each explicit type is numerically approximated by empirical averages over 300 realizations. From the simulation, the numerical result also shows consistency with the theoretical result in Theorem \ref{thm_random 1s0b} and it is illustrated in Figure \ref{Fig: random spread rate}.

\begin{figure}[H] 
	\centering 
	\includegraphics[width=\textwidth]{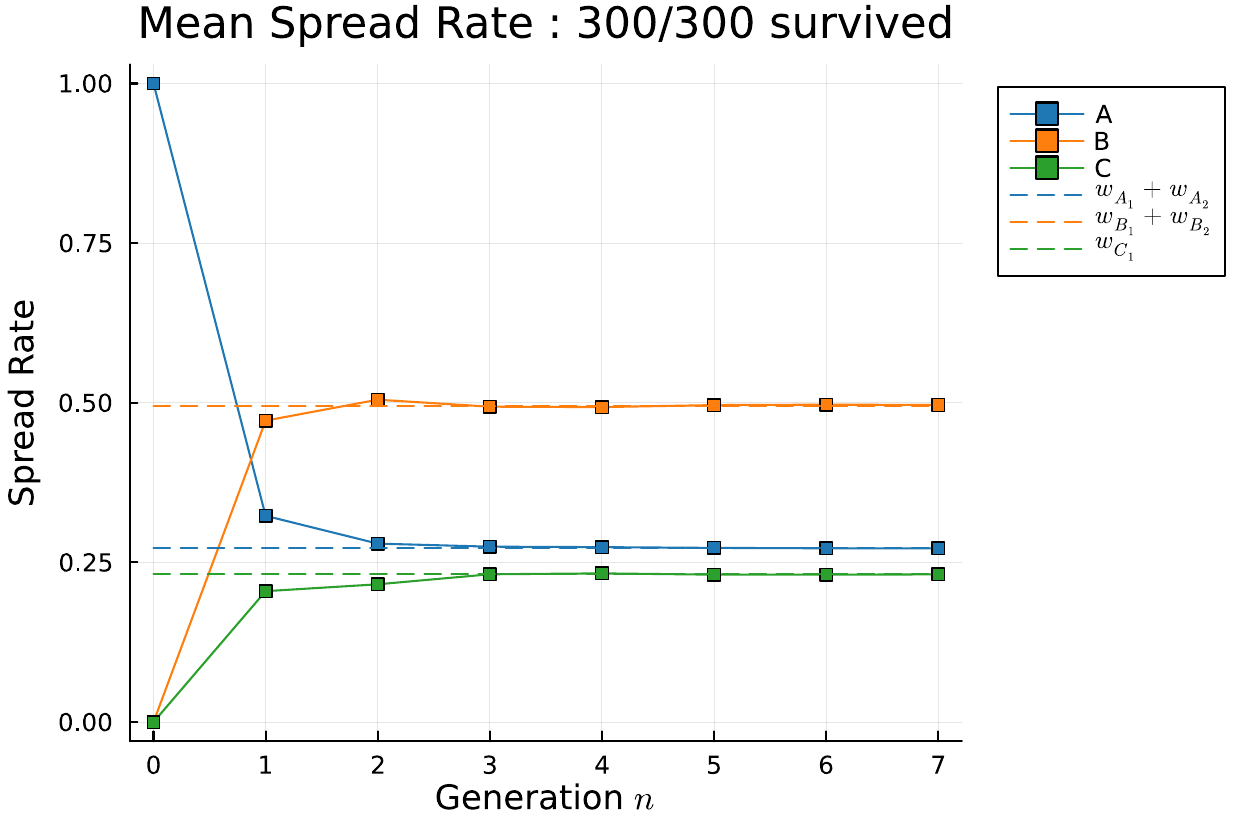} 
	\caption{Examples of a random model for 0-block code. The solid line represents experimental values, while the dashed line represents theoretical values. In this case, the spread rate is numerically approximated by averaging over 300 simulations.} 
  \label{Fig: random spread rate}
\end{figure}

\subsubsection{The $1$-block code case}
Here, we consider the random 1-spread model $\{\mathbf{Z}_n\}_{n\geq 0}$ in Examples \ref{ex: r1} and \ref{ex: r2} with type set $\mathcal{B}=\{b_1,b_2\}$ and obtain its spread mean matrix $M$ according to its spread distribution $\{p^{(b_i)}(\cdot)\}_{i=1}^{2}$:
$$
M=\left(\begin{array}{cc}
\frac{3}{5}&\frac{2}{3}\vspace{0.1cm}\\
1&\frac{1}{2}
\end{array}
\right)$$
Then $\{\mathbf{Z}_n\}_{n\geq 0}$ induces a random $1$-spread model $\{\mathbf{Z}^{[1]}_n\}_{n\geq 0}$ with the type set $\mathbf{B}^{[1]}=\{\mathbf{b}_1, \mathbf{b}_2, \mathbf{b}_3, \mathbf{b}_4 \}$, which is the set of all potential $1$-patterns of $\{\mathbf{Z}_n\}_{n\geq 0}$. According to the spread distribution found in Example \ref{ex: r2}, the spread mean matrix of $\{\mathbf{Z}^{[1]}_n\}_{n\geq 0}$ is a $4\times 4$ matrix given by
$$
\mathbf{M}^{[2]}=\left(\begin{array}{cccc}
\frac{1}{3}&\frac{2}{3}&\frac{1}{2}&\frac{1}{2}\vspace{0.1cm}\\
\frac{2}{3}&\frac{4}{3}&\frac{1}{2}&\frac{1}{2}\vspace{0.1cm}\\
\frac{1}{3}&\frac{2}{3}&0&0\vspace{0.1cm}\\
\frac{1}{3}&\frac{2}{3}&\frac{5}{6}&\frac{1}{6}
\end{array}
\right)$$
and its maximal eigenvalue is 
$$
\begin{array}{c}
\mathbf{\rho}^{[1]}\approx 2.22521
\end{array}
$$
with the corresponding normalized left eigenvector
$$
\begin{array}{c}
\mathbf{w}=(\mathbf{w}_1, \mathbf{w}_2, \mathbf{w}_3, \mathbf{w}_4)\approx(0.213874, 0.427749, 0.202534, 0.155843).
\end{array}
$$
Furthermore, consider the $1$-block code $\Phi^{[1]}: \mathcal{P}_1\rightarrow \mathcal{A}=\{a,b\}$ such that
$$
\begin{array}{c}
\Phi^{[1]}(\mathbf{b_1})=\Phi^{[1]}(\mathbf{b_3})=a \hspace{0.5cm} \textrm{ and } \hspace{0.5cm} \Phi^{[1]}(\mathbf{b_2})=\Phi^{[1]}(\mathbf{b_4})=b.
\end{array}
$$
The the spread rates of explicit types in $\mathcal{A}$ in the associated projected spread model $\{\mathbf{Z}^{[1],\Phi}_n\}_{n\geq 0}$ are 
$$
\begin{array}{l}
s_{b}(a_1;\{\mathbf{Z}_n\}_{n\geq 0},\Phi^{[1]})= \mathbf{w}_1+\mathbf{w}_3\approx 0.416408\vspace{0.1cm}\\
s_{b}(a_2;\{\mathbf{Z}_n\}_{n\geq 0},\Phi^{[1]})= \mathbf{w}_2+\mathbf{w}_4\approx 0.583592.
\end{array}
$$
with probability $1$ for any $b\in \mathcal{B}=\{b_1, b_2\}$.

\begin{figure}[H] 
	\centering 
	\includegraphics[width=\textwidth]{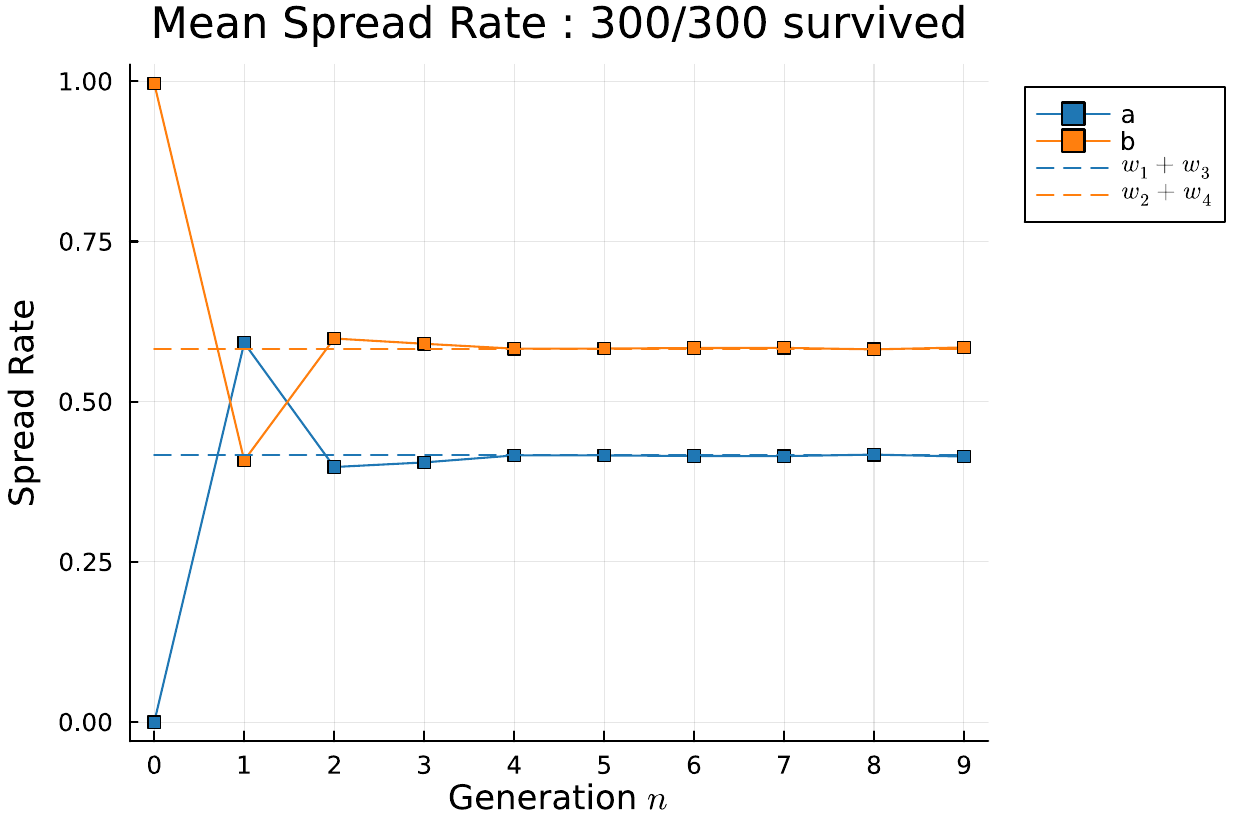} 
	\caption{Examples of a random model for 1-block code. The solid line represents experimental values, while the dashed line represents theoretical values. In this case, the spread rate is numerically approximated by averaging over 300 simulations.}
\label{}
\end{figure}

\subsubsection{The $2$-block code case}

In this case, we will run simulations for the projected spread model $\{\mathbf{Z}^{[2],\Phi}_n\}_{n\geq 0}$ induced from a random $1$-spread model $\{\mathbf{Z}_n\}_{n\geq 0}$ with hidden type set $\mathcal{B}=\{b_1,b_2\}$ and the spread distribution $\{p^{(b_i)}(\cdot)\}_{i=1}^{2}$ such that
$$
\begin{array}{c}
p^{(b_1)}(1,1)=\frac{1}{3},\hspace{0.3cm} p^{(b_1)}(0,1)=\frac{2}{3},\hspace{0.3cm} p^{(b_2)}(1,0)=\frac{1}{2},\hspace{0.3cm} p^{(b_2)}(1,1)=\frac{1}{2},
\end{array}
$$
and the $2$-block code $\Phi^{[2]}: \mathcal{P}_2\rightarrow \mathcal{A}=\{a, b, c, d\}$ such that
$$
\begin{array}{c}
\Phi^{[2]}(\mathbf{b}_1)=\Phi^{[2]}(\mathbf{b}_2)=\Phi^{[2]}(\mathbf{b}_9)=\Phi^{[2]}(\mathbf{b}_{10})=a;\\
\Phi^{[2]}(\mathbf{b}_3)=\Phi^{[2]}(\mathbf{b}_4)=\Phi^{[2]}(\mathbf{b}_{11})=\Phi^{[2]}(\mathbf{b}_{12})=b;\\
\Phi^{[2]}(\mathbf{b}_5)=\Phi^{[2]}(\mathbf{b}_8)=c;\\
\Phi^{[2]}(\mathbf{b}_6)=\Phi^{[2]}(\mathbf{b}_7)=d,
\end{array}
$$
where all the potential $2$-patterns $\mathbf{b}_1, \mathbf{b}_2, \cdots, \mathbf{b}_{12}\in\mathbf{B}^{[2]}$ are listed in Figure \ref{Fig: 2b_patterns b_1} and Figure \ref{Fig: 2b_patterns b_2}. In addition, Figure \ref{Fig: 2b_induced pattern} illustrates how $\{\mathbf{Z}_n\}_{n\geq 0}$ induces the random $1$-spread model $\{\mathbf{Z}^{[2]}_n\}_{n\geq 0}$ with type set $\mathbf{B}^{[2]}=\{\mathbf{b}_1, \cdots, \mathbf{b}_{12}\}$.

\begin{figure}[H] 
	\centering 
	\includegraphics[width=\textwidth]{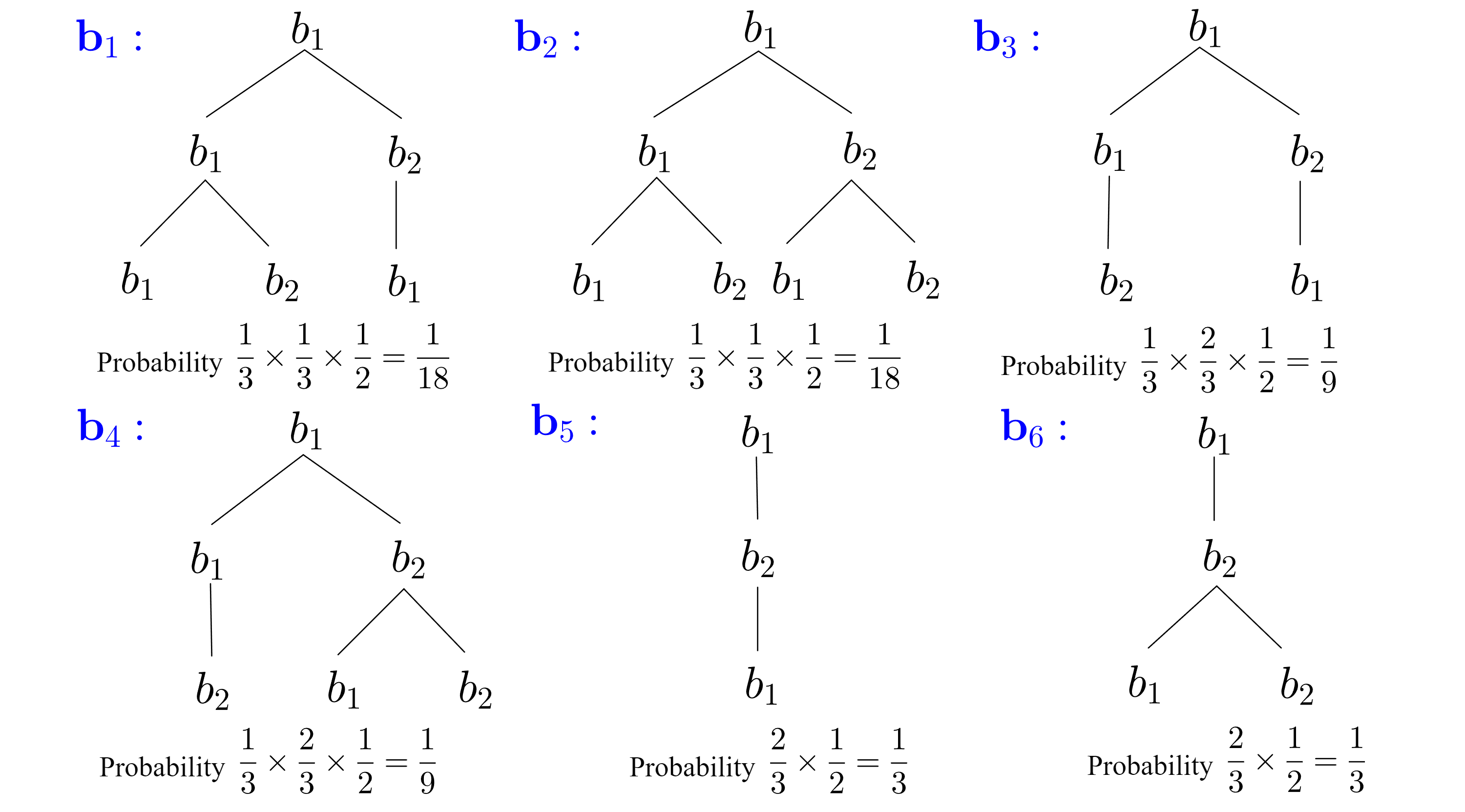} 
	\caption{Potential $2$-patterns initiated with $b_1$} 
\label{Fig: 2b_patterns b_1}
\end{figure}

\begin{figure}[H] 
	\centering 
	\includegraphics[width=\textwidth]{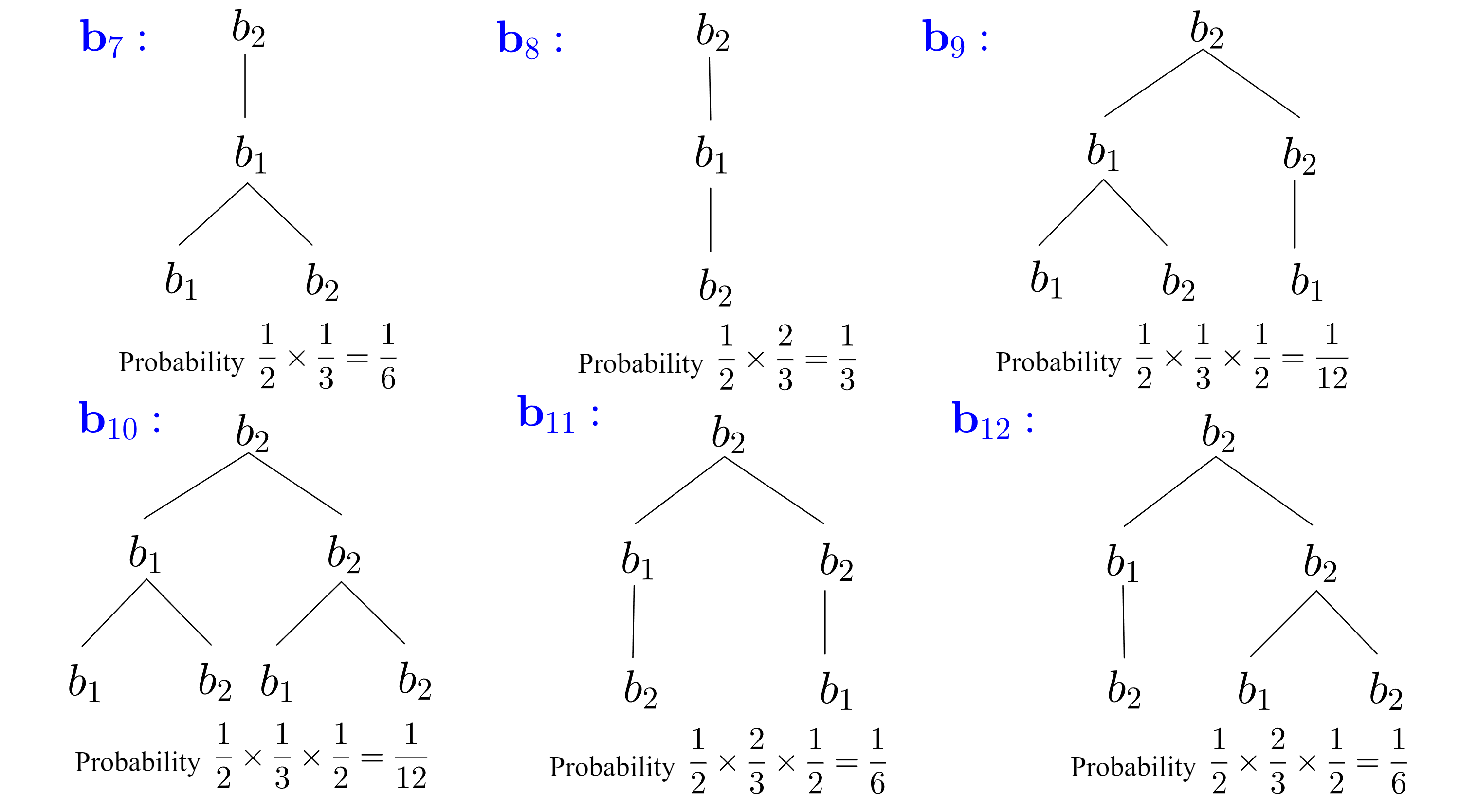} 
	\caption{Potential $2$-patterns initiated with $b_2$} 
\label{Fig: 2b_patterns b_2}
\end{figure}

\begin{figure}[H] 
	\centering 
	\includegraphics[width=\textwidth]{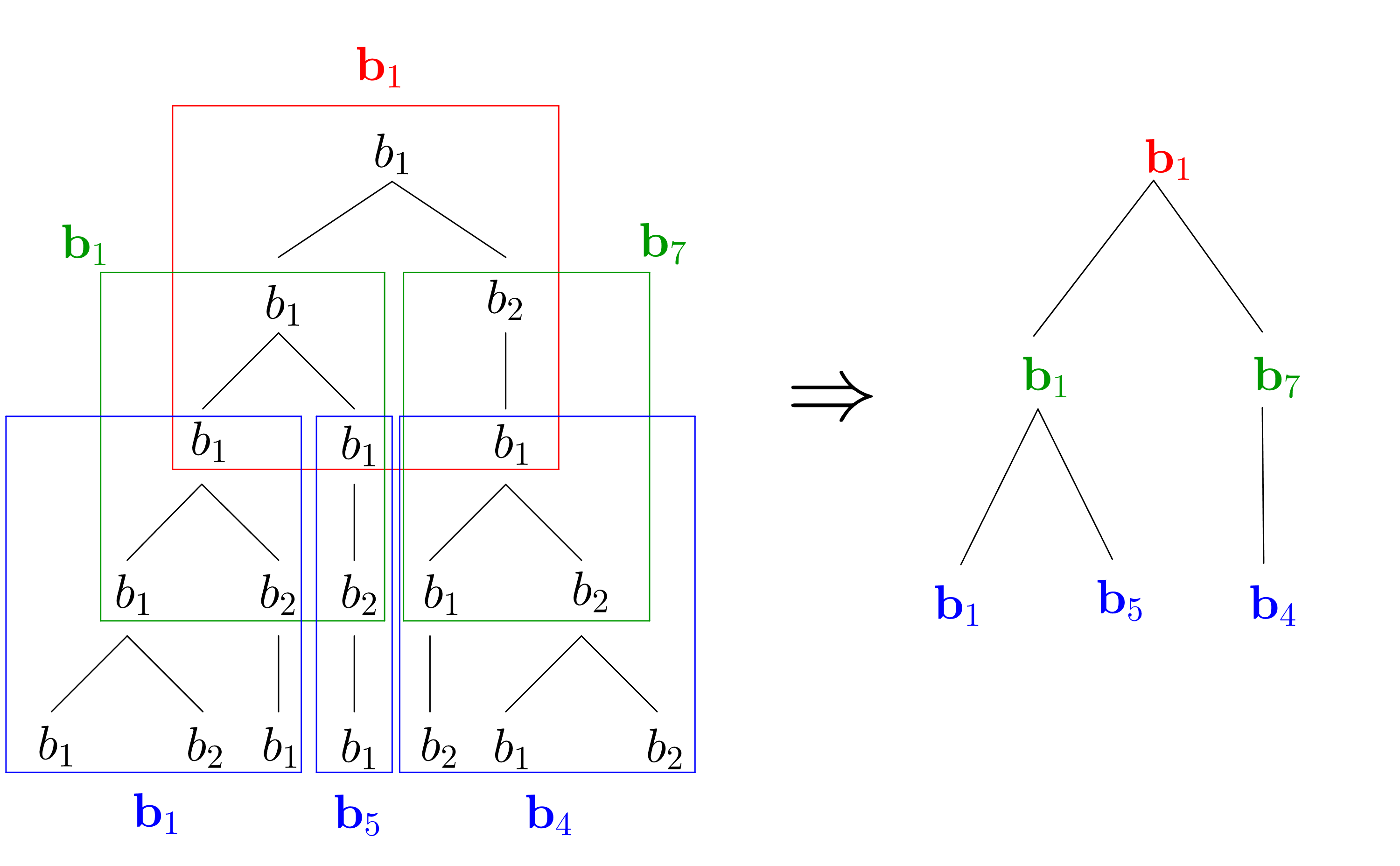} 
	\caption{A $2$-pattern in $\{\mathbf{Z}^{[2]}_n\}_{n\geq0}$ induced from a $4$-pattern in $\{\mathbf{Z}_n\}_{n\geq0}$ }
\label{Fig: 2b_induced pattern}
\end{figure}

The spread distribution $\{p^{(\mathbf{b}_i)}(\cdot)\}_{i=1}^{12}$ of $\{\mathbf{Z}^{[2]}_n\}_{n\geq0}$ can be computed from the spread distribution $\{p^{(b_i)}(\cdot)\}_{i=1}^{2}$ of $\{\mathbf{Z}_n\}_{n\geq0}$ and are listed as follows:
$$
\begin{array}{rl}
&p^{(\mathbf{b}_1)}(1,0,0,0,0,0,1,0,0,0,0,0)=p^{(\mathbf{b}_1)}(0,1,0,0,0,0,1,0,0,0,0,0)=\frac{1}{18},\vspace{0.1cm}\\
&p^{(\mathbf{b}_1)}(0,0,1,0,0,0,1,0,0,0,0,0)=p^{(\mathbf{b}_1)}(0,0,0,1,0,0,1,0,0,0,0,0)\vspace{0.2cm}\\
=&p^{(\mathbf{b}_1)}(1,0,0,0,0,0,0,1,0,0,0,0)=p^{(\mathbf{b}_1)}(0,1,0,0,0,0,0,1,0,0,0,0)=\frac{1}{9},\vspace{0.2cm}\\
&p^{(\mathbf{b}_1)}(0,0,1,0,0,0,0,1,0,0,0,0)=p^{(\mathbf{b}_1)}(0,0,0,1,0,0,0,1,0,0,0,0)=\frac{2}{9};
\end{array}
$$

$$
\begin{array}{rl}
&p^{(\mathbf{b}_2)}(1,0,0,0,0,0,0,0,1,0,0,0)=p^{(\mathbf{b}_2)}(1,0,0,0,0,0,0,0,0,1,0,0)\vspace{0.1cm}\\
=&p^{(\mathbf{b}_2)}(0,1,0,0,0,0,0,0,1,0,0,0)=p^{(\mathbf{b}_2)}(0,1,0,0,0,0,0,0,0,1,0,0)=\frac{1}{36},\vspace{0.2cm}\\
&p^{(\mathbf{b}_2)}(0,0,1,0,0,0,0,0,1,0,0,0)=p^{(\mathbf{b}_2)}(0,0,1,0,0,0,0,0,0,1,0,0)\vspace{0.1cm}\\
=&p^{(\mathbf{b}_2)}(0,0,0,1,0,0,0,0,1,0,0,0)=p^{(\mathbf{b}_2)}(0,0,0,1,0,0,0,0,0,1,0,0)\vspace{0.1cm}\\
=&p^{(\mathbf{b}_2)}(1,0,0,0,0,0,0,0,0,0,1,0)=p^{(\mathbf{b}_2)}(1,0,0,0,0,0,0,0,0,0,0,1)\vspace{0.1cm}\\
=&p^{(\mathbf{b}_2)}(0,1,0,0,0,0,0,0,0,0,1,0)=p^{(\mathbf{b}_2)}(0,1,0,0,0,0,0,0,0,0,0,1)=\frac{1}{18},\vspace{0.2cm}\\
&p^{(\mathbf{b}_2)}(0,0,1,0,0,0,0,0,0,0,1,0)=p^{(\mathbf{b}_2)}(0,0,1,0,0,0,0,0,0,0,0,1)\vspace{0.1cm}\\
=&p^{(\mathbf{b}_2)}(0,0,0,1,0,0,0,0,0,0,1,0)=p^{(\mathbf{b}_2)}(0,0,0,1,0,0,0,0,0,0,0,1)=\frac{1}{9};\\
\end{array}
$$

$$
\begin{array}{c}
p^{(\mathbf{b}_3)}(0,0,0,0,1,0,1,0,0,0,0,0)=p^{(\mathbf{b}_3)}(0,0,0,0,0,1,1,0,0,0,0,0)=\frac{1}{6},\vspace{0.1cm}\\
p^{(\mathbf{b}_3)}(0,0,0,0,1,0,0,1,0,0,0,0)=p^{(\mathbf{b}_3)}(0,0,0,0,0,1,0,1,0,0,0,0)=\frac{1}{3};\\
\end{array}
$$

$$
\begin{array}{rl}
&p^{(\mathbf{b}_4)}(0,0,0,0,1,0,0,0,1,0,0,0)=p^{(\mathbf{b}_4)}(0,0,0,0,0,1,0,0,1,0,0,0)\vspace{0.1cm}\\
=&p^{(\mathbf{b}_4)}(0,0,0,0,1,0,0,0,0,1,0,0)=p^{(\mathbf{b}_4)}(0,0,0,0,0,1,0,0,0,1,0,0)=\frac{1}{12},\vspace{0.2cm}\\
&p^{(\mathbf{b}_4)}(0,0,0,0,1,0,0,0,0,0,1,0)=p^{(\mathbf{b}_4)}(0,0,0,0,0,1,0,0,0,0,1,0)\vspace{0.1cm}\\
=&p^{(\mathbf{b}_4)}(0,0,0,0,1,0,0,0,0,0,0,1)=p^{(\mathbf{b}_4)}(0,0,0,0,0,1,0,0,0,0,0,1)=\frac{1}{6};
\end{array}
$$

$$
\begin{array}{c}
p^{(\mathbf{b}_5)}(0,0,0,0,0,0,1,0,0,0,0,0)=\frac{1}{3},\vspace{0.1cm}\\
p^{(\mathbf{b}_5)}(0,0,0,0,0,0,0,1,0,0,0,0)=\frac{2}{3};\\
\end{array}
$$

$$
\begin{array}{c}
p^{(\mathbf{b}_6)}(0,0,0,0,0,0,0,0,1,0,0,0)=p^{(\mathbf{b}_6)}(0,0,0,0,0,0,0,0,0,1,0,0)=\frac{1}{6},\vspace{0.1cm}\\
p^{(\mathbf{b}_6)}(0,0,0,0,0,0,0,0,0,0,1,0)=p^{(\mathbf{b}_6)}(0,0,0,0,0,0,0,0,0,0,0,1)=\frac{1}{3};\\
\end{array}
$$

$$
\begin{array}{c}
p^{(\mathbf{b}_7)}(1,0,0,0,0,0,0,0,0,0,0,0)=p^{(\mathbf{b}_7)}(0,1,0,0,0,0,0,0,0,0,0,0)=\frac{1}{6},\vspace{0.1cm}\\
p^{(\mathbf{b}_7)}(0,0,1,0,0,0,0,0,0,0,0,0)=p^{(\mathbf{b}_7)}(0,0,0,1,0,0,0,0,0,0,0,0)=\frac{1}{3};\\
\end{array}
$$

$$
\begin{array}{c}
p^{(\mathbf{b}_8)}(0,0,0,0,1,0,0,0,0,0,0,0)=\frac{1}{2},\vspace{0.1cm}\\
p^{(\mathbf{b}_8)}(0,0,0,0,0,1,0,0,0,0,0,0)=\frac{1}{2};\\
\end{array}
$$

$$
\begin{array}{rl}
&p^{(\mathbf{b}_9)}(1,0,0,0,0,0,1,0,0,0,0,0)=p^{(\mathbf{b}_9)}(0,1,0,0,0,0,1,0,0,0,0,0)=\frac{1}{18},\vspace{0.1cm}\\
&p^{(\mathbf{b}_9)}(0,0,1,0,0,0,1,0,0,0,0,0)=p^{(\mathbf{b}_9)}(0,0,0,1,0,0,1,0,0,0,0,0)\vspace{0.2cm}\\
=&p^{(\mathbf{b}_9)}(1,0,0,0,0,0,0,1,0,0,0,0)=p^{(\mathbf{b}_9)}(0,1,0,0,0,0,0,1,0,0,0,0)=\frac{1}{9},\vspace{0.2cm}\\
&p^{(\mathbf{b}_9)}(0,0,1,0,0,0,0,1,0,0,0,0)=p^{(\mathbf{b}_9)}(0,0,0,1,0,0,0,1,0,0,0,0)=\frac{2}{9};
\end{array}
$$

$$
\begin{array}{rl}
&p^{(\mathbf{b}_{10})}(1,0,0,0,0,0,0,0,1,0,0,0)=p^{(\mathbf{b}_{10})}(1,0,0,0,0,0,0,0,0,1,0,0)\vspace{0.1cm}\\
=&p^{(\mathbf{b}_{10})}(0,1,0,0,0,0,0,0,1,0,0,0)=p^{(\mathbf{b}_{10})}(0,1,0,0,0,0,0,0,0,1,0,0)=\frac{1}{36},\vspace{0.2cm}\\
&p^{(\mathbf{b}_{10})}(0,0,1,0,0,0,0,0,1,0,0,0)=p^{(\mathbf{b}_{10})}(0,0,1,0,0,0,0,0,0,1,0,0)\vspace{0.1cm}\\
=&p^{(\mathbf{b}_{10})}(0,0,0,1,0,0,0,0,1,0,0,0)=p^{(\mathbf{b}_{10})}(0,0,0,1,0,0,0,0,0,1,0,0)\vspace{0.1cm}\\
=&p^{(\mathbf{b}_{10})}(1,0,0,0,0,0,0,0,0,0,1,0)=p^{(\mathbf{b}_{10})}(1,0,0,0,0,0,0,0,0,0,0,1)\vspace{0.1cm}\\
=&p^{(\mathbf{b}_{10})}(0,1,0,0,0,0,0,0,0,0,1,0)=p^{(\mathbf{b}_{10})}(0,1,0,0,0,0,0,0,0,0,0,1)=\frac{1}{18},\vspace{0.2cm}\\
&p^{(\mathbf{b}_{10})}(0,0,1,0,0,0,0,0,0,0,1,0)=p^{(\mathbf{b}_{10})}(0,0,1,0,0,0,0,0,0,0,0,1)\vspace{0.1cm}\\
=&p^{(\mathbf{b}_{10})}(0,0,0,1,0,0,0,0,0,0,1,0)=p^{(\mathbf{b}_{10})}(0,0,0,1,0,0,0,0,0,0,0,1)=\frac{1}{9};\\
\end{array}
$$

$$
\begin{array}{c}
p^{(\mathbf{b}_{11})}(0,0,0,0,1,0,1,0,0,0,0,0)=p^{(\mathbf{b}_{11})}(0,0,0,0,0,1,1,0,0,0,0,0)=\frac{1}{6},\vspace{0.1cm}\\
p^{(\mathbf{b}_{11})}(0,0,0,0,1,0,0,1,0,0,0,0)=p^{(\mathbf{b}_{11})}(0,0,0,0,0,1,0,1,0,0,0,0)=\frac{1}{3}\\
\end{array}
$$
and
$$
\begin{array}{rl}
&p^{(\mathbf{b}_{12})}(0,0,0,0,1,0,0,0,1,0,0,0)=p^{(\mathbf{b}_{12})}(0,0,0,0,0,1,0,0,1,0,0,0)\vspace{0.1cm}\\
=&p^{(\mathbf{b}_{12})}(0,0,0,0,1,0,0,0,0,1,0,0)=p^{(\mathbf{b}_{12})}(0,0,0,0,0,1,0,0,0,1,0,0)=\frac{1}{12},\vspace{0.2cm}\\
&p^{(\mathbf{b}_{12})}(0,0,0,0,1,0,0,0,0,0,1,0)=p^{(\mathbf{b}_{12})}(0,0,0,0,0,1,0,0,0,0,1,0)\vspace{0.1cm}\\
=&p^{(\mathbf{b}_{12})}(0,0,0,0,1,0,0,0,0,0,0,1)=p^{(\mathbf{b}_{12})}(0,0,0,0,0,1,0,0,0,0,0,1)=\frac{1}{6}.
\end{array}
$$
Therefore, the corresponding spread mean matrix is given as
$$
\mathbf{M}^{[3]}=\left(\begin{array}{cccccccccccc}
\frac{1}{6}&\frac{1}{6}&\frac{1}{3}&\frac{1}{3}&0&0&\frac{1}{3}&\frac{2}{3}&0&0&0&0\vspace{0.1cm}\\
\frac{1}{6}&\frac{1}{6}&\frac{1}{3}&\frac{1}{3}&0&0&0&0&\frac{1}{6}&\frac{1}{6}&\frac{1}{3}&\frac{1}{3}\vspace{0.1cm}\\
0&0&0&0&\frac{1}{2}&\frac{1}{2}&\frac{1}{3}&\frac{2}{3}&0&0&0&0\vspace{0.1cm}\\
0&0&0&0&\frac{1}{2}&\frac{1}{2}&0&0&\frac{1}{6}&\frac{1}{6}&\frac{1}{3}&\frac{1}{3}\vspace{0.1cm}\\
0&0&0&0&0&0&\frac{1}{3}&\frac{2}{3}&0&0&0&0\vspace{0.1cm}\\
0&0&0&0&0&0&0&0&\frac{1}{6}&\frac{1}{6}&\frac{1}{3}&\frac{1}{3}\vspace{0.1cm}\\
\frac{1}{6}&\frac{1}{6}&\frac{1}{3}&\frac{1}{3}&0&0&0&0&0&0&0&0\vspace{0.1cm}\\
0&0&0&0&\frac{1}{2}&\frac{1}{2}&0&0&0&0&0&0\vspace{0.1cm}\\
\frac{1}{6}&\frac{1}{6}&\frac{1}{3}&\frac{1}{3}&0&0&\frac{1}{3}&\frac{2}{3}&0&0&0&0\vspace{0.1cm}\\
\frac{1}{6}&\frac{1}{6}&\frac{1}{3}&\frac{1}{3}&0&0&0&0&\frac{1}{6}&\frac{1}{6}&\frac{1}{3}&\frac{1}{3}\vspace{0.1cm}\\
0&0&0&0&\frac{1}{2}&\frac{1}{2}&\frac{1}{3}&\frac{2}{3}&0&0&0&0\vspace{0.1cm}\\
0&0&0&0&\frac{1}{2}&\frac{1}{2}&0&0&\frac{1}{6}&\frac{1}{6}&\frac{1}{3}&\frac{1}{3}
\end{array}
\right)$$
and its maximal eigenvalue is $\rho\approx 1.4201325$ and the associated normalized left eigenvector is 
$$
\begin{array}{c}
\mathbf{w}\approx (0.02662,0.02662,0.05324,0.05324,0.159734,0.159734,0.086799,\\0.173599,0.043399,0.043399,0.086799,0.086799).
\end{array}
$$
Figure \ref{Fig: 2-b code} does show the fact that the spread rates of explicit types converge to the sums of the corresponding components of the eigenvector $\mathbf{w}$.

\begin{figure}[H] 
	\centering 
	\includegraphics[width=\textwidth]{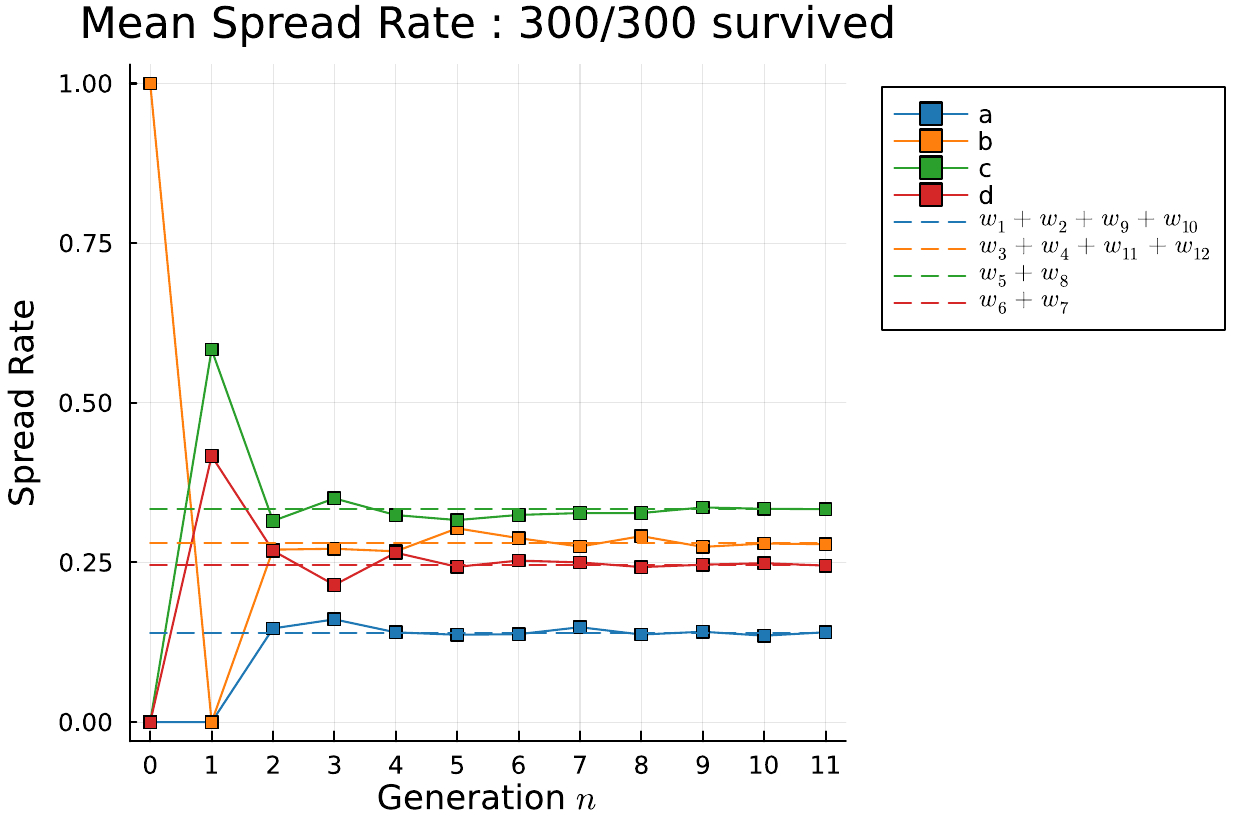} 
	\caption{Examples of a random model for 2-block code. The solid line represents experimental values, while the dashed line represents theoretical values. In this case, the spread rate is numerically approximated by averaging over 300 simulations.}
\label{Fig: 2-b code}
\end{figure}

\section{Conclusions}
When a pandemic occurs, it has a significant impact not only on public health but also on wider society, the economy and so on. Therefore, predicting the development and transmission rate of infectious diseases is crucial for prevention and control. Epidemiological investigations and patient classification are often the first steps in this series of tasks. The effectiveness of patient classification often relies on the accuracy of testing reagents and instruments. Sometimes, due to the urgency and limitations of technological development, test results may generate gray areas, where patients with different attributes may be classified into the same category for certain reasons. In order to describe this phenomenon and provide a method for predicting transmission rates in such situations, we introduce a map and propose the topological and random projected spread models to describe the differences in classification before and after testing.

Given a topological or random spread model that represents the current transmission situation of the disease and a map that projects patterns with hidden types to explicit types, we construct the associated topological or random projected spread model to predict the spread rates associated with explicit types. The significance of this work is that these projected spread models have a wide range of applications. In particular, they can take into consideration the possibility that the contagious behavior of a patient of a certain hidden type may vary during the incubation period. In this case, we may first consider the patient as an individual of some ``subtype" (referred to as a hidden type) according to the stage of the incubation period the patient belongs to and then project these ``subtypes" back to the original type (referred to as an explicit type) using a $0$-block code, so that we can determine the spread rates of the original types in this way. We also generalize this idea of projected spread models to the $m$-spread models together with $k$-block codes for any positive integer $m$ and nonnegative integer $k$, and the spread rates are found in all cases where $m-1=k$, $m-1<k$ and $m-1>k$. In addition, we conduct some simulations in which the results provide numerical evidence supporting our theorems, both in the topological and random cases.

\bibliographystyle{amsplain}
\bibliography{main}

\end{document}